\documentclass{amsart}

\usepackage{amssymb,color}
\usepackage{amsfonts}
\usepackage{amsmath}
\usepackage{euscript}
\usepackage{enumerate}
\usepackage{graphics}

\newtheorem{theorem}{Theorem}[section]
\newtheorem{lemma}[theorem]{Lemma}
\newtheorem{note}[theorem]{Note}
\newtheorem{prop}[theorem]{Proposition}
\newtheorem{cor}[theorem]{Corollary}

\newtheorem*{Theorem1'}{Theorem 1'}

\theoremstyle{definition}

\theoremstyle{remark}

\numberwithin{equation}{section}

\newcommand \g{{\mathfrak g}}

\newcommand  \s{{\mathfrak s}}

\newcommand \tr{{\mathrm {tr}}}

\newcommand \End{{\mathrm {End}}}
\newcommand \dm{{\mathrm {dim}}}

\newcommand \gl{{\mathfrak {gl}}}
\renewcommand \sl{{\mathfrak {sl}}}
\newcommand \so{{\mathfrak {so}}}
\newcommand \sy{{\mathfrak {sp}}}
\newcommand \C{{\mathbb C}}

\newcommand \GL{{\mathrm {GL}}}
\newcommand \la{{\lambda}}

\newcommand \al{{\alpha}}

\newcommand \va{{\varepsilon}}
\newcommand \vp{{\varphi}}

\begin{document}

\title [Composition series of $\gl(m)$ as a module for its classical subalgebras]{Composition series of $\gl(m)$ as a module for its classical subalgebras over an arbitrary field}

\author{Martin Chaktoura}
\address{Department of Mathematics and Statistics, Univeristy of Regina, Canada}
\email{martin\_chaktoura@yahoo.com.ar}

\author{Fernando Szechtman}
\address{Department of Mathematics and Statistics, Univeristy of Regina, Canada}
\email{fernando.szechtman@gmail.com}
\thanks{The second author was supported in part by an NSERC discovery grant}





\begin{abstract}
Let $F$ be an arbitrary field and let $f:V\times V\to F$ be a
non-degenerate symmetric or alternating bilinear form defined on
an $F$-vector space of finite dimension $m\geq 2$. Let $L(f)$ be
the subalgebra of $\gl(V)$ formed by all skew-adjoint
endomorphisms with respect to $f$. We find a composition series
for the $L(f)$-module $\gl(V)$ and furnish multiple
identifications for all its composition factors.
\end{abstract}

\maketitle

\section{Introduction}

Let $F$ be any field. No assumptions are made on $F$ or its
characteristic, which will be denoted by $\ell$. Let $V$ be an
$F$-vector space of finite dimension $m\geq 2$, and let $f:V\times
V\to F$ be a non-degenerate symmetric or alternating bilinear
form. Consider the subalgebra $L(f)$ of $\gl(V)$ defined by
$$
L(f)=\{x\in\gl(V)\,|\, f(xv,w)=-f(v,xw)\text{ for all }v,w\in V\}.
$$
Thus $L(f)$ is the symplectic Lie algebra if $f$ is alternating,
or an orthogonal Lie algebra if $f$ is symmetric and
non-alternating (the last condition is only required if $\ell=2$).
Note that, in general, the isomorphism type of an orthogonal Lie
algebra depends on the equivalence type of the underlying form,
and we will only speak of {\em the} orthogonal Lie algebra when
$F=F^2$, i.e., when every element of $F$ is a square. We further
let $\s=Z(\gl(V))$, which consists of all scalar operators.

In this paper we find a composition series for the $L(f)$-module
$\gl(V)$ and furnish multiple identifications for all its
composition factors. All possible cases are considered, without
exception. Numerous cases arise, as all of $\ell,F,f$ and $m$ play
a role in the determination of the structure of $\gl(V)$. Our main
results are as follows.

\begin{theorem}\label{intro31}
Suppose that $\ell=2$, that $m=2n$, and that $f$ is non-degenerate
and alternating. Then

(1) If $4\mid m$ then the $L(f)$-module $\gl(V)$ has $m+6$
composition factors. A composition series can be obtained by
inserting $m-1$ arbitrary subspaces between $L(f)$ and
$L(f)^{(1)}=[L(f),L(f)]$ in the series
$$
0\subset \s\subset  L(f)^{(2)}\subset  L(f)^{(1)}\subset
L(f)\subset U \subset \sl(V)\subset \gl(V),
$$
where $U=L(f)\oplus\langle x\rangle$, $x\in\sl(V)$,
$[x,L(f)]\subseteq L(f)$, and
$$
x=\left(%
\begin{array}{cc}
  I_n & 0\\
  0 & 0 \\
\end{array}%
\right).
$$
All composition factors are trivial, except for
$L(f)^{(2)}/\s\cong \sl(V)/U$, which is of dimension
${{m}\choose{2}}-2$. Moreover, $L(f)^{(2)}/\s$ is a simple Lie
algebra if and only if $m>4$.

(2) If $m\neq 2$ and $4\nmid m$ then the $L(f)$-module $\gl(V)$
has $m+4$ composition factors.  A composition series can be
obtained by inserting $m-1$ arbitrary subspaces between $L(f)$ and
$L(f)^{(1)}$ in the series
$$
0\subset  L(f)^{(2)}\subset  L(f)^{(1)}\subset L(f)\subset
\sl(V)\subset \gl(V).
$$
All composition factors are trivial, except for $L(f)^{(2)}\cong
\sl(V)/L(f)$. Moreover, $L(f)^{(2)}$ is a simple Lie algebra of
dimension ${{m}\choose{2}}-1$.

(3) $L(f)$ is isomorphic to the symmetric square $S^2(V)$ as
$L(f)$-modules, and, relative to suitable basis of $V$, consists
of all matrices
\begin{equation}\label{maint}
\left(\begin{array}{cc}
   A & B\\
   C & A'
  \end{array}
\right),\, A,B,C\in\gl(n),\text{ where }B,C\text{ are symmetric}.
\end{equation}

(4) $L(f)^{(1)}$ is isomorphic to the exterior square
$\Lambda^2(V)$ as $L(f)$-modules, and consists of all matrices
(\ref{maint}) such that $B,C$ are alternating.

(5) $L(f)^{(2)}$ is isomorphic to the kernel of the contraction
$L(f)$-epimorphism $\Lambda^2(V)\to F$, given by $v\wedge w\mapsto
f(v,w)$, and consists of all matrices (\ref{maint}) such that
$B,C$ are alternating and $\tr(A)=0$.

(6) $L(f)/L(f)^{(2)}$ is isomorphic, as Lie algebra, to
${\mathfrak h}(n)$, the Heisenberg algebra of dimension $2n+1$.
\end{theorem}

The case $m=4$ of Theorem \ref{intro31} is exceptional in various
ways. Firstly, while $L(f)^{(2)}/\s$ is an irreducible
$L(f)$-module, it is not simple as a Lie algebra. A similar
phenomenon occurs to $L(f)^{(1)}$ if $m=4$, $\ell=2$ but $f$ is
non-degenerate, symmetric and non-alternating, as discussed below.
Secondly, $L(f)^{(2)}$ is also isomorphic, as Lie algebra, to
${\mathfrak h(2)}$, so $L(f)$ is an extension of ${\mathfrak
h(2)}$ by ${\mathfrak h(2)}$. Thirdly, the kernel of the
representation of $L(f)$ on $L(f)^{(2)}/\s$ is $L(f)^{(2)}$. This
gives a 4-dimensional faithful irreducible representation of
${\mathfrak h}(2)$. This phenomenon is impossible in
characteristic not 2. Full details of the case $m=4$, as well as
its connections to the problem of finding the smallest dimension
of a faithful module for a given Lie algebra (see \cite{Bu} and
\cite{CR}) can be found in \S\ref{lastsec}, which also treats the
much easier case $m=2$.

We remark that Bourbaki \cite{B}, Chapter I, \S 6, Exercise 25(b),
studied the ideal structure of $L(f)$ when $\ell=2$ and $f$ is
non-degenerate and alternating, but made mistakes involving $\s$,
$L(f)^{(1)}$ and $L(f)^{(2)}$ (see Note \ref{errorb} for details).
Theorem~\ref{intro31} corrects Bourbaki's information and expands
it to include the structure of $\gl(V)$ as $L(f)$-module, as well
as providing further identifications for all composition factors.

In the case of orthogonal Lie algebras in characteristic 2, which
was not considered in \cite{B}, we have the following result.

\begin{theorem}\label{intro41} Suppose that $\ell=2$, that $m\geq 2$, and that $f$ is
non-degenerate, symmetric and non-alternating. Then

(1) The $L(f)$-module $\gl(V)$ has $m+2$ composition factors. A
composition series can be obtained by inserting $m-1$ arbitrary
subspaces between $L(f)$ and $L(f)^{(1)}$ in the series
$$
0\subset L(f)^{(1)}\subset L(f)\subset \gl(V).
$$
Moreover, if $m=3$ or $m\geq 5$ then $L(f)^{(1)}$ is a simple Lie
algebra of dimension~${{m}\choose{2}}$.

(2) $L(f)$ is isomorphic to the symmetric square $S^2(V)$ as
$L(f)$-modules. Moreover, there is a basis of $V$ relative to
which $f$ has Gram matrix $D=\mathrm{diag}(d_1,\dots,d_m)$ and,
relative to this basis, $L(f)$ consists of all $A\in\gl(m)$ such
that $d_iA_{ij}=d_j A_{ji}$.

(3) $L(f)^{(1)}$ is isomorphic to the exterior square
$\Lambda^2(V)$ as $L(f)$-modules. Moreover, relative to the above
basis, $L(f)^{(1)}$ consists of all $A\in\gl(m)$ such that
$A_{ii}=0$ and $d_iA_{ij}=d_j A_{ji}$.

(4) $\gl(V)/L(f)\cong L(f)^{(1)}$ as $L(f)$-modules. In
particular, $\gl(V)$ has $m$ trivial composition factors, and 2
composition factors isomorphic to $L(f)^{(1)}\cong \Lambda^2(V)$,
which is itself the trivial module if and only if $m=2$.
\end{theorem}

We remark that if we let $m=4$ in Theorem \ref{intro41} then the
irreducible $L(f)$-module $L(f)^{(1)}$ need not be a simple Lie
algebra. The structure of the 6-dimensional Lie algebra
$L(f)^{(1)}$ depends on whether the discriminant of $f$ is a
square in $F$ or not. This is entirely analogous to what happens
to $L(f)$ itself when $\ell\neq 2$ and $f$ is non-degenerate and
symmetric, as indicated in \cite{B}, Chapter I, \S 6, Exercise
26(b). For a uniform treatment of both cases via current Lie
algebras see \cite{CS}.

Our results in characteristic not 2 are better described by means
of
$$
M(f)=\{y\in\gl(V)\,|\, f(yv,w)=f(v,yw)\text{ for all }v,w\in V\}.
$$
Note that $M(f)$ is an $L(f)$-module, regardless of the nature of
$f$ and $\ell$. However, if $\ell=2$ then $M(f)=L(f)$, so $M(f)$
plays no additional role in this case.

\begin{theorem} \label{intro51} Suppose that $\ell\neq 2$, that $m=2n$, and that $f$ is
non-degenerate and skew-symmetric. Then

(1) $M(f)$ is the orthogonal complement to $L(f)$ with respect to
the bilinear form $\vp:\gl(V)\times \gl(V)\to F$, given by
$\vp(x,y)=\tr(xy)$. Moreover, $M(f)$ consists, relative to
suitable basis of $V$, of all matrices
\begin{equation}
\label{maint2} \left(\begin{array}{cc}
   A & B\\
   C & A'
  \end{array}
\right),\, A,B,C\in\gl(n),\text{ where }B,C\text{ are
skew-symmetric}.
\end{equation}
Furthermore, $M(f)$ is isomorphic to $\Lambda^2(V)$ as
$L(f)$-module.

(2) $M(f)\cap\sl(V)$ consists of all matrices (\ref{maint2}) such
that $\tr(A)=0$ and is isomorphic to the kernel of the contraction
$L(f)$-epimorphism $\Lambda^2(V)\to F$ given by $v\wedge w\to
f(v,w)$.

(3) If $m>2$ and $\ell\nmid m$ then $M(f)\cap\sl(V)$ is an
irreducible $L(f)$-module of dimension ${{m}\choose{2}}-1$.

(4) If $m>2$ and $\ell\mid m$ then $M(f)\cap\sl(V)/\s$ is an
irreducible $L(f)$-module of dimension ${{m}\choose{2}}-2$.

(5) $L(f)$ is a simple Lie algebra, isomorphic to both
$\gl(V)/M(f)$ and $S^2(V)$ as $L(f)$-modules.

(6) The following are composition series of the $L(f)$-module
$\gl(V)$:
$$
0\subset \s\subset M(f)\cap\sl(V)\subset M(f)\subset \gl(V),\text{
if }m>2\text{ and }\ell|m,
$$
$$
0\subset M(f)\cap\sl(V)\subset M(f)\subset \gl(V),\text{ if
}m>2\text{ and }\ell\nmid m,
$$
$$
0\subset M(f)\subset \gl(V),\text{ if }m=2.
$$
In any case, $M(f)/M(f)\cap\sl(V)$ is the trivial $L(f)$-module.
\end{theorem}

\begin{theorem} \label{intro61} Suppose that $\ell\neq 2$, that $m\geq 2$, and that $f$ is
non-degenerate and symmetric. Then

(1) $M(f)$ is the orthogonal complement to $L(f)$ with respect to
the bilinear form $\vp:\gl(V)\times \gl(V)\to F$, given by
$\vp(x,y)=\tr(xy)$. Moreover, there is a basis of $V$ relative to
which $f$ has Gram matrix $D=\mathrm{diag}(d_1,\dots,d_n)$ and,
relative to this basis, $M(f)$ consists of all $A\in\gl(m)$ such
that $d_iA_{ij}=d_j A_{ji}$. Furthermore, $M(f)$ is isomorphic to
$S^2(V)$ as $L(f)$-module.

(2) $M(f)\cap\sl(V)$ consists, relative to the above basis, of all
matrices $A\in\gl(m)$ such that $d_iA_{ij}=d_j A_{ji}$ and
$\tr(A)=0$, and is isomorphic to the kernel of the contraction
$L(f)$-epimorphism $S^2(V)\to F$ given by $vw\to f(v,w)$.

(3) If $m\geq 4$ and $\ell\nmid m$ then $M(f)\cap\sl(V)$ is an
irreducible $L(f)$-module of dimension~${{m+1}\choose{2}}-1$.

(4) If $m\geq 4$ and $\ell\mid m$ then $M(f)\cap\sl(V)/\s$ is an
irreducible $L(f)$-module of dimension ${{m+1}\choose{2}}-2$.

(5) If $m=3$ or $m\geq 5$ then $L(f)$ is a simple Lie algebra,
isomorphic to both $\gl(V)/M(f)$ and $\Lambda^2(V)$ as
$L(f)$-modules.

(6) The following are composition series of the $L(f)$-module
$\gl(V)$:
$$
0\subset \s\subset M(f)\cap\sl(V)\subset M(f)\subset \gl(V),\text{
if }m\geq 4\text{ and }\ell|m,
$$
$$
0\subset M(f)\cap\sl(V)\subset M(f)\subset \gl(V),\text{ if }m\geq
4\text{ and }\ell\nmid m.
$$
In any case, $M(f)/M(f)\cap\sl(V)$ is the trivial $L(f)$-module.
\end{theorem}

As remarked earlier, if we take $m=4$ in Theorem \ref{intro61},
the structure of Lie algebra $L(f)$ depends on the nature of the
discriminant of $f$. On the other hand, if we take $m=2$ or $m=3$
in Theorem \ref{intro61}, the structure of the $L(f)$-module
$M(f)$ depends not just on $\ell$ but also on $F$ itself. See
\S\ref{dondeso} for details.

Much is known about the classical Lie algebras and their
representations, so a great deal of the results stated above is
already known. Indeed, note that $f$ induces an isomorphism
$V\cong V^*$ of $L(f)$-modules, so $\gl(V)\cong V^*\otimes V\cong
V\otimes V$, where $V\otimes V/S^2(V)\cong \Lambda^2(V)$, with
$V\otimes V=S^2(V)\oplus \Lambda^2(V)$ if $\ell\neq 2$. Let
$\Omega:V\otimes V\to F$ be the contraction $L(f)$-epimorphism
given by $v\otimes w\mapsto f(v,w)$.

Suppose $F=\C$. If $f$ is skew-symmetric and $m\geq 4$ then
$V\otimes V$ has the following decomposition into irreducible
$L(f)$-submodules:
$$
V\otimes V=S^2(V)\oplus (\ker\Omega\cap \Lambda^2(V))\oplus U,
$$
where $S^2(f)\cong L(f)$ and $U$ is trivial. If $f$ is symmetric,
with $m=3$ or $m\geq 5$, then $V\otimes V$ decomposes as follows
into irreducible $L(f)$-submodules:
$$
V\otimes V=\Lambda^2(V)\oplus (\ker\Omega\cap S^2(V))\oplus W,
$$
where $\Lambda^2(f)\cong L(f)$ and $W$ is trivial. We refer the
reader to \cite{FH} for these details, as well as for further
information, in terms of Weyl modules, on higher tensor powers of
$V$.

It follows from Theorems \ref{intro31}-\ref{intro61} that the
above statements remain valid for any field of characteristic 0,
but cease to be true if $\ell|2m$, the more substantial failure
occurring when $\ell=2$. In prime characteristic, the ideal
structure of $L(f)$ is described in detail in \cite{B}, Chapter I,
\S 6, Exercises 25 and 26, although $\ell\neq 2$ is required in
the orthogonal case. Given the mistakes found in \cite{B} and that
full information on the $L(f)$-submodule structure of $\gl(V)$,
that includes all possible cases of $\ell,F,f$ and $m$, does not
seem to be available in the literature, we decided to provide a
self-contained account of it, including complete proofs, and
requiring no prior knowledge of Lie algebras.

We begin in \S\ref{se-pre}, which includes all one needs to know
about Lie algebras to read this paper. This material can be
covered during the first week of a course on the subject.

We now refer to \cite{H}, \S 1, Exercise 10, where we are required
to justify the complex isomorphisms $\sy(4)\cong\so(5)$ and
$\sl(4)\cong\so(6)$. How is one supposed to prove this after a
single week of lecturing? Comparing multiplication tables is one
option, although very tiring and time consuming, as these Lie
algebras have dimensions 10 and 15, respectively. The use of
Dynkin diagrams must postponed until much later, so one is
essentially led to use representations in some way or another. In
the case of $\sl(4)\cong\so(6)$ there is the standard argument
involving the action of $\sl(4)$ on $\Lambda^2(W)$, where $W$ is
the natural module of $\sl(4)$. From $\sl(4)\cong\so(6)$ one then
obtains $\sy(4)\cong\so(5)$ by restriction to $\sy(4)$. This
requires prior knowledge of exterior powers, which might not be
available to everyone at the beginning (or the end) of a course on
Lie algebras, especially to undergraduate students.

In \S\ref{sec1} we furnish an extremely elementary and direct
proof of $\sy(4)\cong\so(5)$ whenever $\ell\neq 2$ and $F=F^2$
(these conditions are clearly necessary) as part of a general and
canonical imbedding $\sy(2n)\hookrightarrow \so(2n^2-n-1)$
whenever $\ell\nmid 2n$ and $F=F^2$. The material from
\S\ref{sec1} is really a special case of our study of the
$L(f)$-module $M(f)$ in the symplectic case, but we present it
first to make the isomorphism $\sy(4)\cong\so(5)$ available
immediately after the first rudiments on Lie algebras. If we had
to single out a key ingredient behind the isomorphism
$\sy(4)\cong\so(5)$ it would be the non-degenerate
$\gl(V)$-invariant symmetric bilinear form~$\vp:\gl(V)\times
\gl(V)\to F$ used in Theorems \ref{intro51} and \ref{intro61}.

In \S\ref{sec2} we give an elementary proof of $\sl(4)\cong\so(6)$
valid when $\ell\neq 2$ and $F=F^2$ (these conditions are, again,
necessary). It uses the same idea of the isomorphism
$\sy(4)\cong\so(5)$, the presence of a non-degenerate invariant
symmetric bilinear form, although in this case a minimal amount of
calculations are needed in order to avoid the use of exterior
powers. As with classical method, the simplicity of $\sl(4)$ is
required. For completeness, an account of the ideal structure of
$\gl(m)$ is given in~\S\ref{sec12}. This can be found in \cite{B},
Chapter I, \S 6, Exercise 24.

In \S\ref{secbas} we describe basic properties of $M(f)$ for an
arbitrary bilinear form $f$, with emphasis on the case when $f$ is
non-degenerate, symmetric or alternating, while \S\ref{secid}
justifies the various identifications made in Theorems
\ref{intro31}-\ref{intro61} concerning $L(f)$-modules.

The last five sections are devoted to demonstrate the
irreducibility aspects of Theorems~\ref{intro31}-\ref{intro61},
depending on whether $\ell=2$ or not and the nature of $f$.

\section{Preliminaries}\label{se-pre}

The notation introduced in this section will be maintained
throughout the entire paper.

Let $F$ be an arbitrary field of characteristic $\ell$. Thus
$\ell$ is zero or a prime. All vector spaces are assumed to be
finite dimensional over $F$ unless otherwise mentioned. We fix a
vector space $V$ of dimension $m\geq 1$.

\subsection{Lie algebras} A  Lie algebra is a vector space $L$ together with a bilinear map
$[~,~]:L\times L\to L$, called bracket or commutator, satisfying:

(L1) $[x,x]=0$ for all $x\in L$

(L2) $[x,[y,z]]+[y,[z,x]]+[z,[x,y]]=0$ for all $x,y,z\in L$.

Any associative algebra $A$ gives rise to a Lie algebra whose
underlying vector space is $A$ itself, with commutator
$$
[xy]=xy-yx,\quad x,y\in A.
$$
The Lie algebras corresponding to $M_m(F)$ and $\End(V)$ will be
denoted by $\gl(m)$ and $\gl(V)$, respectively, and called general
linear Lie algebras.

The canonical matrices $e_{ij}$, $1\leq i,j\leq m$, form a basis
of $M_m(F)$ and multiply as follows:
$e_{ij}e_{kl}=\delta_{jk}e_{il}$. Thus, we have the following
multiplication table in $\gl(m)$:
\begin{equation}
\label{mult} [e_{ij},e_{kl}]=\delta_{jk}e_{il}-\delta_{il}e_{kj}.
\end{equation}

Given Lie algebras $L_1$ and $L_2$, a Lie homomorphism (resp.
isomorphism) is a linear homomorphism (resp. isomorphism)
$T:L_1\to L_2$ satisfying $$T([x,y])=[T(x),T(y)],\quad x,y\in
L_1.$$ For instance, if $\mathcal B$ is a basis of $V$ then the
map $M_{\mathcal B}:\gl(V)\to \gl(m)$, which sends every
$x\in\gl(V)$ to its matrix $M_{\mathcal B}(x)$ relative to
${\mathcal B}$, is a Lie isomorphism.

Given a Lie algebra $L$, an ideal (resp. subalgebra) is a subspace
$K$ of $L$ satisfying $[x,y]\in K$ for all $x\in L$ and $y\in K$
(resp. all $x,y\in K$). We say that $L$ is simple if $\dim(L)>1$
and the only ideals of $L$ are $0$ and $L$.

For instance, $L^{(1)}=[L,L]$, the span of all $[x,y]$ with
$x,y\in L$, is an ideal of $L$, as well as
$L^{(2)}=[L^{(1)},L^{(1)}]$, etc.

The special linear Lie algebra $\sl(V)$, consisting of all
traceless endomorphisms of~$V$, is an ideal of $\gl(V)$,
corresponding to $\sl(m)$, the ideal of $\gl(m)$ of all traceless
$m\times m$ matrices, under the isomorphism $M_{\mathcal B}$.

We will denote by $\s$ the ideal of $\gl(V)$ (resp. $\gl(m)$) of
all scalar endomorphisms (resp. matrices). Note that $\s\subseteq
\sl(V)$ if and only if $\ell|m$.

\subsection{Representations and modules}\label{subre} Let $L$ be a Lie algebra.
A representation of $L$ on a vector space $W$ is a Lie
homomorphism $R:L\to\gl(W)$, in which case we refer to $W$ as an
$L$-module and write $x\cdot w$ or simply $xw$ to mean $R(x)w$.
Note that the map $L\times W\to W$ is bilinear and satisfies
\begin{equation}\label{axmod}
[x,y]w=xyw-yxw,\quad x,y\in L, w\in W.
\end{equation}
Conversely, any bilinear map map $L\times W\to W$ satisfying
(\ref{axmod}) gives rise to a representation $R:L\to\gl(W)$
defined by $R(x)w=xw$.

Let $W$ be an $L$-module. We say that $W$ is faithful if its
associated representation is injective. An $L$-submodule of $W$ is
a subspace $U$ of $W$ such that $xu\in U$ for all $x\in L$ and
$u\in U$. We refer to $W$ as irreducible if $W$ is non-zero and
its only submodules are 0 and $W$. For instance, the adjoint
module of a Lie algebra $L$ is $W=L$, where $x\cdot w=[x,w]$. This
is irreducible if and only if $L$ is a simple Lie algebra or
$\dim(L)=1$.

Note that the dual space $W^*$ becomes an $L$-module via
$$(x\cdot \alpha)(w)=\alpha(-x\cdot w),\quad x\in L,\,\alpha\in W^*,\, w\in W.$$
Using annihilators we easily see that $W$ is irreducible if and
only if so is $W^*$.

Let $R:L\to\gl(W)$ and $R^*:L\to\gl(W^*)$ be the representations
associated to $W$ and $W^*$. Let ${\mathcal C}$ be a basis of $W$
and ${\mathcal C}^*$ is its dual basis. Then the matrix
representations associated to $W$ and $W^*$ with respect to
${\mathcal C}$ and ${\mathcal C}^*$ are related by:
$$
M_{{\mathcal C}^*}(R^*(x))=-M_{{\mathcal C}}(R(x))',\quad x\in L,
$$
where $A'$ denotes the transpose a matrix $A$.

A homomorphism (resp. isomorphism) of $L$-modules is a linear
homomorphism (resp. isomorphism) $T:W_1\to W_2$ satisfying
$$
T(xw)=xT(w),\quad x\in L,w\in W_1.
$$

\subsection{Classical Lie algebras}\label{subcla} We fix throughout a bilinear form $f:V\times V\to F$ and set
$$
L(f)=\{x\in\gl(V)\,|\, f(xv,w)=-f(v,xw)\text{ for all }v,w\in V\},
$$
$$
M(f)=\{y\in\gl(V)\,|\, f(yv,w)=f(v,yw)\text{ for all }v,w\in V\}.
$$
Note that $L(f)=M(f)$ if $\ell=2$.
\begin{lemma}
\label{lem11} $L(f)$ is a subalgebra of $\gl(V)$ and $M(f)$ is an
$L(f)$-submodule of~$\gl(V)$.
\begin{proof} Let $x\in L(f)$, $y\in M(f)$. We wish to see that $[xy]\in M(f)$. If $v,w\in V$ then
$$
\begin{aligned}
f([xy]v,w)&=f(xyv,w)-f(yxv,w)\\
&=-f(yv,xw)-f(xv,yw)\\
&=-f(v,yxw)+f(v,xyw)\\
&=f(v,[xy]w).
\end{aligned}
$$
The proof that $L(f)$ is a subalgebra of $\gl(V)$ is entirely
analogous.
\end{proof}
\end{lemma}

It will be useful to have a matrix version of $L(f)$ and $M(f)$
available. Given $A\in\gl(m)$, we set
\begin{equation}\label{mjo}
L(A)=\{X\in\gl(m)\,|\, X'A=-AX\}\text{ and }M(A)=\{Y\in\gl(m)\,|\,
Y'A=AY\}.
\end{equation}

Let ${\mathcal B}=\{v_1,\dots,v_m\}$ be a basis of $V$ and suppose
that $A\in\gl(m)$ is the Gram matrix of $f$ relative to ${\mathcal
B}$, that is,
$$
A_{ij}=f(v_i,v_j),\quad 1\leq i,j\leq m.
$$
Then $M_{\mathcal B}$ sends $L(f)$ onto $L(A)$ and $M(f)$
onto $M(A)$.

Two matrices $A,B\in\gl(m)$ are said to be congruent if there is
$S\in\GL_m(F)$ such that
$$
S'AS=B,
$$
in which case the map $L(A)\to L(B)$ given by $X\mapsto S^{-1}XS$
is a Lie isomorphism.

Suppose $f$ is non-degenerate and alternating. In this case (see
\cite{K}, Theorem 19) $m=2n$ and there is a basis ${\mathcal B}$
of $V$ relative to which $f$ has Gram matrix
\begin{equation}\label{defjo}
J=\left(%
\begin{array}{cc}
  0 & I_n \\
  -I_n & 0 \\
\end{array}%
\right).
\end{equation}
We write $\sy(2n)=L(J)$ and refer to $L(f)$ as the symplectic Lie
algebra. An easy computation based on (\ref{mjo}) and (\ref{defjo}) reveals that
$$L(J)=\left\{\left(\begin{array}{cc}
   A & B\\
   C & -A'
  \end{array}
\right)\,\big |\,A,B,C\in\gl(n),\text{ where }B,C\text{ are
symmetric}\right\}
$$
and
$$M(J)=\left\{\left(\begin{array}{cc}
   A & B\\
   C & A'
  \end{array}
\right)\,\big |\,A,B,C\in\gl(n),\text{ where }B,C\text{ are
skew-symmetric}\right\}.
$$
In particular,
$$
\dim L(J)={{m+1}\choose{2}},\text{ and if }\ell\neq 2\text{ then }
\dim M(J)={{m}\choose{2}}.
$$

Suppose next that $f$ is non-degenerate and symmetric. If $f$ is
alternating then necessarily $\ell=2$ and $L(f)$ is the symplectic
Lie algebra considered above. If $f$ is non-alternating then by
\cite{K}, Theorems 18 and 20, there is a basis ${\mathcal B}$ of
$V$ relative to which $f$ has diagonal Gram matrix
$D=\mathrm{diag}(d_1,\dots,d_m)$, $d_i\neq 0$. Another calculation
based on (\ref{mjo}) shows that
$$
L(D)=\{A\in\gl(m)\,|\, d_iA_{ij}+d_jA_{ji}=0\text{ for all } 1\leq
i,j\leq m\},
$$
and
$$
M(D)=\{A\in\gl(m)\,|\, d_iA_{ij}-d_jA_{ji}=0\text{ for all } 1\leq
i,j\leq m\}.
$$
In particular, if $f$ is non-degenerate, symmetric and
non-alternating then
$$
\dim L(f)={{m}\choose{2}}\text{ if }\ell\neq 2,\text{ }\dim
L(f)={{m+1}\choose{2}}\text{ if }\ell=2,
$$
and
$$
\dim M(f)={{m+1}\choose{2}}.
$$
Moreover, if $F=F^2$ (i.e., every element of $F$ is a square) then
$f$ admits $I_m$ as Gram matrix, in which case we refer to $L(f)$
as the orthogonal Lie algebra and write $\so(m)=L(I_m)$. Clearly,
$\so(m)$ consists of all skew-symmetric matrices of~$\gl(m)$ and
$M(I_m)$ of all symmetric matrices of $\gl(m)$.

A matrix $A\in\gl(m)$ is said to be alternating if
$$A_{ij}=-A_{ji}\text{
and  }A_{ii}=0,\quad 1\leq i,j\leq m.
$$
By above, any two invertible alternating matrices are congruent.
Provided $F=F^2$, so are any two invertible symmetric
non-alternating matrices.

Suppose that $W$ is a $d$-dimensional module for a Lie algebra
$L$, where $d\geq 1$. A bilinear form $\phi:W\times W\to F$ is
said to be $L$-invariant if the representation $R:L\to\gl(W)$
associated to $W$ satisfies $R:L\to L(\phi)$, that is, if
$$
\phi(x\cdot u,v)+\phi(u,x\cdot v)=0,\quad x\in L,\, u,v\in W.
$$
In this case, if the $L$-module $W$ is faithful and $\phi$ is
non-degenerate we obtain an imbedding $R:\g\to\sy(d)$ (resp.
$R:\g\to\so(d)$) provided $\phi$ is alternating (resp. symmetric
and non-alternating, and $F=F^2$).

Let $T:W\to W^*$ be a linear map. By definition, $T$ is an
$L$-homomorphism if and only if the associated bilinear form
$\phi:V\times V\to F$, given by $\phi(u,v)=T(u)(v)$, is
$L$-invariant, in which case $T$ is an isomorphism if and only if
$\phi$ is non-degenerate.

For instance, if $f$ is non-degenerate then $V\cong V^*$ as
$L(f)$-modules via the map $v\mapsto f(v,-)$, since $f$ is
$L(f)$-invariant.

\subsection{A trace form}\label{subtra} We fix throughout the bilinear form $\vp:\gl(V)\times \gl(V)\to F$
given by
\begin{equation}
\label{tra} \vp(x,y)=\tr(xy),\quad x,y\in\gl(V).
\end{equation}
It is well-known and easy to see that $\vp$ is symmetric and
non-degenerate.

\begin{lemma}
\label{lem12} The bilinear form $\vp$ is $\gl(V)$-invariant.
\begin{proof} Let $x,y,z\in\gl(V)$. Then
$$
\varphi(z\cdot x,y)+\varphi(x,z\cdot
y)=\tr([zx]y)+\tr(x[zy])=\tr(zxy-xzy)+\tr(xzy-xyz)=0.
$$
\end{proof}
\end{lemma}
By abuse of notation we will also denote by $\vp$ the
$\gl(m)$-invariant non-degenerate symmetric bilinear form
$\gl(m)\times \gl(m)\to F$ defined by $\vp(A,B)=\tr(AB)$. We have
$$
\s=\sl(m)^\perp,\quad \sl(m)=\s^\perp.
$$
Moreover, if $\ell\nmid m$ then
$$
\gl(m)=\sl(m)\perp\s.
$$

We let $\mathrm{Alt}(m)$ and $\mathrm{Sym}(m)$ stand for the
spaces of all alternating and symmetric $m\times m$ matrices,
respectively. Consider the linear map
$\Psi:\gl(m)\to\mathrm{Alt}(m)$ given by $A\mapsto A-A'$. Since
$\ker(\Psi)=\mathrm{Sym}(m)$, the rank-nullity formula implies
that $\mathrm{im}(\Psi)=\mathrm{Alt}(m)$, i.e., $\Psi$ is
surjective.

Observe next that if $A,B\in \gl(m)$ then
\begin{equation}
\label{trazz} \tr(AB)-\tr(A'B')=\tr(AB)-\tr(B' A')= \tr(AB)-\tr((A
B)')=0.
\end{equation}
Suppose $C\in \mathrm{Alt}(m)$ and $B\in \mathrm{Sym}(m)$. Since
$\Psi$ is surjective, we have $C=A-A'$ for some $A\in\gl(m)$, so
by (\ref{trazz})
\begin{equation}
\label{skew} \vp(C,A)=\tr(CB)=\tr((A-A')B)=\tr(AB-A'B')=0.
\end{equation}
Combining (\ref{skew}) with the non-degeneracy of $\vp$,
dimension considerations show that
$$
\mathrm{Alt}(m)^\perp=\mathrm{Sym}(m),\quad
\mathrm{Sym}(m)^\perp=\mathrm{Alt}(m).
$$
Moreover, if $\ell\neq 2$ then
$$
\gl(m)=\mathrm{Alt}(m)\perp\mathrm{Sym}(m).
$$

\subsection{Weights} Suppose $H$ and $W$ are vector spaces and that $H$ acts on $W$,
i.e., there is a bilinear map $H\times W\to W$, say $(h,w)\mapsto
hw$. Then every $\al\in H^*$ gives rise to the subspace, say
$W_\al$, of $W$ defined by
$$W_\al=\{w\in W\,|\, h w=\al(h)w\text{ for all }h\in H\}.$$
We say that $\al$ is weight for the action of $H$ on $W$ if
$W_\al\neq 0$. Note that if $T:W\to W'$ is an isomorphism of
$L$-modules for a Lie algebra $L$ and $H$ is a subalgebra of $L$
then $T(W_\al)=W'_\al$ for every $\al\in H^*$. In particular, the
weights for the actions of $H$ on $W$ and $W'$ are identical.

\begin{note}{\rm As an illustration, consider the irreducible $\sl(V)$-module
$V$ and the diagonal subalgebra $H$ of $\sl(V)$. The weights of
$H$ acting on $V$ are $\varepsilon_1,\dots,\varepsilon_m$, where
$\varepsilon_i:H\to F$ is the $i$th coordinate function, given by
$\varepsilon_i(h)=h_{ii}$. The weights of $H$ acting on~$V^*$ are
$-\varepsilon_1,\dots,-\varepsilon_m$. Thus, if $m>2$ and
$\ell\neq 2$ then $V\not\cong V^*$. If $m=2$ then $H$ has
the same weights $\varepsilon_1,-\varepsilon_1$ acting on
$V$ and $V^*$ and, in fact, $V\cong V^*$. If $m>2$ but $\ell=2$
then $H$ has the same weights on $V$ and $V^*$. However, in this
case $V\not\cong V^*$, otherwise $V\otimes V\cong \gl(V)$, which
contradicts the $\sl(V)$-submodule structures of $V\otimes V$ and
$\gl(V)$.

An alternative way to decide when $V\cong V^*$ is to
look at the automorphism $A\mapsto -A'$ of $\sl(m)$. It is given
by conjugation by a fixed $S\in\GL_m(F)$ if and only if $m\leq 2$.

The above phenomenon when $\ell=2$ is impossible for $F=\C$: an
irreducible module for a complex semisimple Lie algebra is
characterized by the weights of a Cartan subalgebra.}
\end{note}

\section{Viewing $\gl(2n)$ as a module for $\sy(2n)$}
\label{sec1}

We assume throughout this section that $\ell\neq 2$ and $m=2n$,
and set $W=\gl(2n)$. Recalling the matrix $J\in\gl(2n)$ defined in
(\ref{defjo}), we also set $L=L(J)=\sy(2n)$. Note that $W$ is an
$L$-module via $x\cdot w=[x,w]$. Recall, as well, the non-degenerate
$L$-invariant symmetric bilinear form $\vp:W\times W\to F$,
defined in (\ref{tra}), and the $L$-submodule $M=M(J)$, defined in
(\ref{mjo}).

\begin{theorem}
\label{purr} Suppose that $\ell\nmid 2n$. Then the $L$-module $W$
has the following orthogonal decomposition into $L$-submodules:
\begin{equation}
\label{lea1} W=L\perp (M\cap\sl(2n))\perp\s,
\end{equation}
where $M=L^\perp$ has the matrix description given in
\S\ref{subcla}.

Moreover, $M\cap\sl(2n)$ is an $L$-module of dimension $2n^2-n-1$,
which is faithful if $n\geq 2$. In particular, if $F=F^2$ and
$n\geq 2$ then $M\cap\sl(2n)$ induces an imbedding
$\sy(2n)\hookrightarrow \so(2n^2-n-1)$, which is an isomorphism
$\sy(4)\to\so(5)$ when $n=2$.
\end{theorem}

\begin{proof} We claim that $M=L^\perp$. Indeed, let $x\in M$
and $y\in L$. As indicated in \S\ref{subcla}, there exist
$a,b,c,d,e\in\gl(n)$ such that
\begin{equation}
\label{diab} x=\left(\begin{array}{cc}
   a & b\\
   c & a'
  \end{array}
\right),\, y=\left(\begin{array}{cc}
   d & e\\
   f & -d'
  \end{array}
\right),\text{ with }b,c\text{ skew-symmetric }\text{ and }
e,f\text{ symmetric}.
\end{equation}
It follows from (\ref{trazz}) and (\ref{skew}) that
$$
\vp(x,y)=\tr(xy)=\tr(ad+bf+ce-a'd')=0.
$$
This proves $M\subseteq L^\perp$. The matrix descriptions of $L$
and $M$ show that $W=L\oplus M$. On the other hand, the
non-degeneracy of $\varphi$ implies $\dm\, L+\dm\,L^\perp=\dm\,
W$. Since $M\subseteq L^\perp$ and they have the same dimension,
it follows that $M=L^\perp$. We have shown
\begin{equation}
\label{diab2} W=L\perp M.
\end{equation}
The matrix description of $M$ makes it clear that
$$\dim M\cap \sl(2n)=2n^2-n-1,$$
and the condition $\ell\nmid 2n$ implies
\begin{equation}
\label{diab3} M=(M\cap\sl(2n))\perp\s.
\end{equation}
Substituting (\ref{diab3}) into (\ref{diab2}) yields (\ref{lea1}).

Since $\sl(2n)$ and $\s$ are ideals of $\gl(2n)$, it follows from
Lemma \ref{lem11} that all components of (\ref{lea1}) are
$L$-submodules of $W$. Since $\vp$ is symmetric and non-degenerate
on $W$, so its restriction to each component of (\ref{lea1}). As
explained in \S\ref{subcla}, this yields an imbedding
$\sy(2n)\hookrightarrow \so(2n^2-n-1)$, provided $F=F^2$ and
$M\cap \sl(2n)$ is a faithful $L$-module. This imbedding becomes
an isomorphism $\sy(4)\to\so(5)$ when $n=2$, as these Lie algebras
are both 10-dimensional.

It only remains to show that $M\cap \sl(2n)$ is a faithful
$L$-module whenever $n\geq 2$. For this purpose, let $y\in L$ be
as in (\ref{diab}) and suppose that $[x,y]=0$ for all $x\in M\cap
\sl(2n)$ as in (\ref{diab}). Setting $b=0=c$, it follows that
$[d,a]=0$ for all $a$ in $\gl(n)$, whence $d$ is scalar. Letting
$a=0=c$, we see that $2db=0$ and $bf=0$ for all skew-symmetric
$b\in\gl(n)$, so $d=0=f$. Finally, taking $a=0=b$, we get $ce=0$
for all skew-symmetric $c\in\gl(n)$, whence $e=0$. This completes
the proof.
\end{proof}

\begin{note}{\rm The irreducibility of the components of (\ref{lea1}) is discussed
in \S\ref{donde}.

Here we sketch an indirect argument of the irreducibility and
faithfulness of $M\cap\sl(2n)$ when $n\geq 2$ and $F=\C$.

Let $H$ be the diagonal subalgebra of $L$. For $1\leq i\leq n$
consider the linear functional $\varepsilon_i:H\to F$ given by
$\varepsilon_i(h)=h_{ii}$. We easily verify that the weights for
the action of $H$ on $M\cap\sl(2n)$ are the sums of 2 distinct
members taken from
$\{\varepsilon_{1},\dots,\varepsilon_{n},-\varepsilon_{1},\dots,-\varepsilon_{n}\}$.
These are the same weights for the action of $H$ on $V(\la_2)$,
where $\la_2$ is the second fundamental module for $L$. But
$$\dm(M \cap\sl(2n))=2n^2-n-1=\dm(V(\la_2)),$$
so $M \cap\sl(2n)\cong V(\la_2)$ is irreducible.

Since $L$ is a simple Lie algebra and $M \cap\sl(2n)$ is an
irreducible $L$-module of dimension $>1$, it follows that $M
\cap\sl(2n)$ is faithful.}
\end{note}

\section{Viewing $\gl(m)$ as a module for $\sl(m)$}
\label{sec12}

We assume throughout this section that $m\geq 2$.

\begin{theorem}\label{simplysl2}
Suppose $(m,\ell) \ne (2,2)$. Then the only composition series of
$\gl(m)$ as a module for $\gl(m)$ or $\sl(m)$ are:
$$
0\subset \s\subset\sl(m)\subset \gl(m),\text{ if }\ell\mid m,
$$
$$
0\subset \sl(m)\subset \gl(m),\text{ if }\ell\nmid m.
$$
In particular,
$\mathfrak{sl}(m)/\mathfrak{s}\cap\mathfrak{sl}(m)$ is always
simple. More explicitly,
\begin{itemize}
\item if $\ell\nmid m$, then $\mathfrak{sl}(m)$ simple;

\item if $\ell\mid m$, then the only proper non-trivial ideal of
$\mathfrak{sl}(m)$ is $\mathfrak{s}$.
\end{itemize}
\end{theorem}

\begin{proof} Let $I$ be a subspace of $\gl(m)$ invariant under
$\gl(m)$ or $\sl(m)$ and properly containing $\sl(m)\cap\s$. It
suffices to show that $I$ contains $\sl(m)$.

If $e_{ij}\in I$ for some $i\neq j$ then (\ref{mult}) yields that
all $e_{kl}$, with $k\neq l$, as well as all traceless diagonal
matrices, are in $I$, so $\sl(m)\subseteq I$. If some non-scalar
diagonal matrix $h$ is in $I$, then $h_i\neq h_j$ for some $i\neq
j$, so $[h,e_{ij}]=(h_i-h_j)e_{ij}\in I$, and the first case
applies. Suppose $x\in I$ and $x_{ij}\ne 0$ for some $i\ne j$.
Then either $\ell\neq 2$, so $[e_{ji},[e_{ji},x]]=-2x_{ij}e_{ji}$,
and the first case applies, or $m>2$ and there exists
$k\in\{1,\ldots,n\}-\{i,j\}$, so
$[e_{ki},[e_{jk},[e_{ji},x]]]=x_{ij}e_{ji}$, and the first case
apples.
\end{proof}

\begin{note}{\rm It is stated in \cite{B}, Chapter I, \S 6,
Exercise 24(a), that bracketing any non-scalar element of $\gl(m)$
with at most four suitable chosen elements produces a non-zero
scalar multiple of one of the $e_{ij}$. The proof of Theorem
\ref{simplysl2} shows that three elements already
suffice.}
\end{note}

\section{Viewing $\gl(2n)$ as a module for $\sl(n)\oplus\sl(n)$ and $\sl(n)$}
\label{sec2}

If $L_1,L_2$ are Lie algebras then the vector space $L_1\oplus
L_2$ becomes a Lie algebra via
$[x_1+x_2,y_1+y_2]=[x_1,y_1]+[x_2,y_2]$ for $x_1,y_1\in L_1$ and
$x_2,y_2\in L_2$.

Let $W=\gl(r+n)$ be the adjoint module for $\gl(r+n)$. By means of
the imbedding $\gl(r)\oplus \gl(n)\hookrightarrow \gl(r+n)$, given
by $a+b\mapsto a\oplus b$, we may view $W$ as a module for
$\gl(r)\oplus\gl(n)$. We have
\begin{equation}
\label{excal} \left[ \left(\begin{array}{cc}
    a & 0 \\
   0 & b
  \end{array}
\right),\left(\begin{array}{cc}
    0 & s \\
   0 & 0
  \end{array}
\right)\right]=\left(\begin{array}{cc}
    0 & as-sb\\
   0 & 0
  \end{array}
\right)
\end{equation}
and
\begin{equation}
\label{excal2} \left[ \left(\begin{array}{cc}
    a & 0 \\
   0 & b
  \end{array}
\right),\left(\begin{array}{cc}
    0 & 0 \\
   t & 0
  \end{array}
\right)\right]=\left(\begin{array}{cc}
    0 & 0\\
   bt-ta & 0
  \end{array}
\right).
\end{equation}
Thus $Z=M_{r\times n}(F)$ is a $\gl(r)\oplus\gl(n)$-module under
the action
$$(a+b)\cdot s=as-sb,\quad a\in\gl(r),b\in\gl(n),s\in Z$$
and $A=M_{n\times r}(F)$ is a $\gl(r)\oplus\gl(n)$-module under
the action
$$(a+b)\cdot t=bt-ta, \quad a\in\gl(r),b\in\gl(n),s\in A.
$$

\begin{theorem}
\label{az}
 The map $\phi:A\to Z^*$, given by $$\phi_t(s)=\tr(ts),\quad t\in A,s\in Z,$$
 is an isomorphism of $\gl(r)\oplus \gl(n)$-modules.
\end{theorem}

\begin{proof} This is a linear
isomorphism. Moreover, if $a+b\in\gl(r)\oplus\gl(n)$ then
$$
\phi_{(a+b)\cdot t}(s)=\phi_{bt-ta}(s)=\tr(bts-tas),
$$
$$
((a+b)\cdot \phi_t)(s)=\phi_t(-(a+b)\cdot
s)=\phi_t(-as+sb)=\tr(-tas+tsb).
$$
\end{proof}

We assume $r=n$ for the remainder of this section. By
means of the imbedding $\gl(n)\hookrightarrow
\gl(n)\oplus\gl(n)\hookrightarrow \gl(2n)$, given by
$a\mapsto a\oplus -a'$, we may view $W$ as a module for
$\gl(n)$. We have
\begin{equation}
\label{tuy} \left[ \left(\begin{array}{cc}
    a & 0 \\
   0 & -a'
  \end{array}
\right),\left(\begin{array}{cc}
    0 & s \\
   0 & 0
  \end{array}
\right)\right]= \left(\begin{array}{cc}
    0 & as+sa'\\
   0 & 0
  \end{array}
\right)
\end{equation}
and
\begin{equation}
\label{tuy2} \left[ \left(\begin{array}{cc}
    a & 0 \\
   0 & -a'
  \end{array}
\right),\left(\begin{array}{cc}
    0 & 0 \\
   t & 0
  \end{array}
\right)\right]= \left(\begin{array}{cc}
    0 & 0\\
   -a't-ta & 0
  \end{array}
\right).
\end{equation}
Thus $Z=\gl(n)$ becomes a $\gl(n)$-module under the action
$$a\cdot s=as+sa',\quad a\in\gl(n),\,s\in Z$$
and
$A=\gl(n)$ becomes a $\gl(n)$-module under the action
$$
a\cdot t=-a't-ta,\quad a\in\gl(n),\,t\in A.
$$
This is nothing but the automorphism $a\mapsto-a'$ followed the previous action on $Z$.

Clearly the spaces of symmetric and alternating matrices are
$\gl(n)$-submodules of $Z$ (resp. $A$), denoted by $S$ and $T$
(resp. $B$ and $C$). Moreover, if $\ell\neq 2$ we have $Z=S\oplus
T$ (resp. $A=B\oplus C$).

\begin{prop}
\label{az4} Suppose that $\ell\neq 2$. Then the map $\phi:A\to
Z^*$ defined in Theorem \ref{az} is an isomorphism of
$\gl(n)$-modules sending $B$ onto $S^*$ and $C$ onto~$T^*$.
\end{prop}

\begin{proof} It follows from Theorem \ref{az} that $\phi$ is an
isomorphism of $\gl(n)$-modules. Let $b\in B$ and suppose that
$\phi_b(s)=0$ for all $s\in S$. By (\ref{skew}) we also have
$\phi_b(t)=0$ for all $t\in T$. Since $Z=S\oplus T$, it follows
that $b=0$. Dimension considerations imply that $\phi$ sends $B$
onto $S^*$. Likewise we show that $\phi$ sends $C$ onto $T^*$.
\end{proof}

\begin{theorem}
\label{uca}
 Suppose $\ell\neq 2$ and $n=4$. Then the map $h:T\to C$, given by
$$s= \left(\begin{array}{cccc}
    0 & a & b & c\\
   -a& 0 & d & e\\
   -b & -d & 0 & f\\
   -c & -e & -f & 0\\
  \end{array}
\right)\mapsto s^*=\left(\begin{array}{cccc}
    0 & f & -e & d\\
   -f & 0 & c & -b\\
   e & -c & 0 & a\\
   -d & b & -a & 0\\
  \end{array}
\right),$$ is an isomorphism of $\sl(4)$-modules. The composite
map $T\to C\to T^*$, given by $s\mapsto \phi_{s^*}$, is an
isomorphism of $\sl(4)$-modules. The corresponding non-degenerate
$\sl(4)$-invariant bilinear form $g:T\times T\to F$, namely
$g(s,t)=\tr(s^* t)$, is symmetric. Consequently,
$\sl(4)\cong\so(6)$ provided $F=F^2$.
\end{theorem}

\begin{proof} Clearly $h$ is a linear isomorphism. We easily verify that $h$ commutes with the actions of $e_{12},e_{23},e_{34}$ on $T$ and $C$.
In light of (\ref{mult}), the same happens to all $e_{ij}$, where
$1\leq i<j\leq 4$. Let $t\in T$. Then $t=s^*$ for $s=t^*$. If
$1\leq i<j\leq 4$ then $h(e_{ij}\cdot s)=e_{ij}\cdot h(s)$, i.e.,
$(e_{ij}s+se_{ji})^*=-(e_{ji}s^*+s^*e_{ij})$, which means
$-(e_{ij}t^*+t^*e_{ji})=(e_{ji}t+te_{ij})^*$, that is,
$e_{ji}\cdot h(t)=h(e_{ji}\cdot t)$. Using (\ref{mult}) once more
yields that $f$ commutes with the action of all $x\in\sl(4)$. The
symmetry of $g$ is easily verified.

We know from Theorem \ref{simplysl2} that $\sl(4)$ is simple and
it is clear from (\ref{tuy}) that $\sl(4)$ does not act trivially
on $T$. Thus, as explained in \S\ref{subcla}, $g$ yields an
imbedding $\sl(4)\hookrightarrow \so(6)$, which is an isomorphism
since they are both of dimension 15.
\end{proof}

Note that if $n=1$ then $T=0$ and if $n=2$ then $T$ is the trivial
$\sl(2)$-module. We assume $n\geq 2$ for the remainder of this
section.

\begin{theorem}
\label{24} Suppose $\ell\neq 2$. If $n\neq 2,4$ then $T$ is not a
self-dual $\sl(n)$-module.
\end{theorem}

\begin{proof} Let $H$ be the space of diagonal matrices $a\oplus -a'$ with $a\in\sl(n)$.
Let $\va_i:H\to F$ the $i$th coordinate function, $h\mapsto
h_{ii}$, for $1\leq i\leq n$. Observe that
\begin{equation}
\label{uva} a_1\va_1+\cdots+a_n \va_n=0\Leftrightarrow
a_1=\cdots=a_{n}.
\end{equation}
The eigenvalues of $H$ acting on $T$ can explicitly computed from
(\ref{tuy}). They are $\va_i+\va_j$, where $1\leq i<j\leq n$. On
the other hand, (\ref{tuy2}) shows that the eigenvalues of $H$
acting on $C$ are $-(\va_{p}+\va_{q})$, where $1\leq p<q\leq n$.
Thus the sets of eigenvectors for the actions of $H$ on $T$ are
$C$ are disjoint if $n\neq 2,4$ by (\ref{uva}). It follows from
Proposition \ref{az4} that $T\not\cong T^*$.
\end{proof}

\begin{theorem}
\label{ugas} Both $C$ and $T$ are irreducible $\sl(n)$-modules.
\end{theorem}

\begin{proof} Since $a\mapsto -a'$ is an automorphism of $\sl(n)$,
it is clear that $C$ is irreducible if and only if so is $T$. We
next verify that $T$ is irreducible.

Let $a\in\sl(n)$ and set $b=a'$. Suppose $t\in T$. Then
$$a\cdot t=at+ta'=at-(at)'=tb-(tb)'.$$
Thus the $\sl(n)$-module generated by $t$ contains all matrices
obtained from $t$ by arbitrary left and right multiplication by
$\sl(n)$ followed by ``alternation". By doing this we can easily
pass from any $t\neq 0$ to $e_{12}-e_{21}$ and from there to any
$e_{ij}-e_{ij}$.
\end{proof}

\begin{theorem}
\label{ugas33} Suppose $\ell\neq 2$. Then both $B$ and $S$ are
irreducible $\sl(n)$-modules.
\end{theorem}

\begin{proof} By Proposition \ref{az4} we have $B\cong S^*$, so
$B$ is irreducible if and only if so is~$S$. We next verify that
$S$ is irreducible.

Let $a\in\sl(n)$ and set $b=a'$. Suppose $s\in S$. Then $$a\cdot
s=as+(as)'=sb+(sb)'.$$ Thus the $\sl(n)$-module generated by $s$
contains all matrices obtained from $s$ by arbitrary left and
right multiplication followed by ``symmetrization". By doing this
we can easily pass from any $s\neq 0$ to $e_{11}$ and from there
to any $e_{ij}+e_{ij}$.
\end{proof}

\begin{note}{\rm Let $\mathrm{Bil}(V)$ be the vector space all bilinear
forms $\beta:V\times V\to F$. Then $\mathrm{Bil}(V)$ becomes a
$\gl(V)$-module via
\begin{equation}
\label{biliboy} (x\cdot \beta)(u,v)=-\beta(xu,v)-\beta(u,xv).
\end{equation}
Given $\beta\in\mathrm{Bil}(V)$, the subalgebra of all
$x\in\gl(V)$ such that $x\cdot \beta=0$ is just $L(\beta)$.

Let $\mathrm{Sym}(V)$ and $\mathrm{Alt}(V)$ be the subspaces of
symmetric and alternating bilinear forms on $V$. Clearly
$\mathrm{Sym}(V)$ and $\mathrm{Alt}(V)$ are $\gl(V)$-submodules of
$\mathrm{Bil}(V)$. We have canonical $\gl(V)$-isomorphisms
$$
\mathrm{Bil}(V)\cong (V\otimes V)^*\cong V^*\otimes V^*,
$$
mapping $\mathrm{Sym}(V)$ onto the symmetric square $S^2(V^*)$ and
$\mathrm{Alt}(V)$ onto the exterior square $\Lambda^2(V^*)$. Now
(\ref{tuy2}) and (\ref{biliboy}) make it clear that
$\mathrm{Bil}(V)\cong A$. Thus
$$
B\cong S^2(V^*), C\cong \Lambda^2(V^*),
$$
and, if $\ell\neq 2$, then
$$
S\cong S^2(V), T\cong \Lambda^2(V).
$$
Regardless of $\ell$, if $n=4$ we consider the map
$\Lambda^2(V)\times \Lambda^2(V)\to\Lambda^4(V)$ given by
\begin{equation}
\label{weg} (v_1\wedge v_2,v_3\wedge v_4)\mapsto v_1\wedge
v_2\wedge v_3\wedge v_4.
\end{equation}
Since $\dim\Lambda^4(V)=1$, (\ref{weg}) yields a non-degenerate
$\sl(V)$-invariant symmetric bilinear form on $\Lambda^2(V)$. This
form is alternating if $\ell=2$. Thus, if $F=F^2$ we obtain an
isomorphism $\sl(4)\cong \so(6)$ when $\ell\neq 2$ and an
imbedding $\sl(4)\hookrightarrow \sy(6)$ if $\ell=2$. }
\end{note}

\section{Basic Properties of $M(f)$}\label{secbas}

Recall the definition of the bilinear form $\vp:\gl(V)\times
\gl(V)\to F$ given in (\ref{tra}).

\begin{lemma}
\label{lem13} If $L$ is a Lie subalgebra of $\gl(V)$ and $M$ is an
$L$-submodule of $\gl(V)$ then
$$
M^\perp=\{x\in\gl(V)\,|\, \vp(M,x)=0\}
$$
is an $L$-submodule of $\gl(V)$.
\begin{proof} Let $z\in L$, $x\in M^\perp$ and $y\in M$. Then, by
Lemma \ref{lem12}, we have
$$
\vp([zx],y)=-\vp(x,[zy])=0.
$$
\end{proof}
\end{lemma}

\begin{cor}
\label{cor11} The space $L(f)^\perp$ is an $L(f)$-submodule of
$\gl(V)$.
\begin{proof} This follows from Lemmas \ref{lem11} and
\ref{lem13}.
\end{proof}
\end{cor}

For a subspace $U$ of $V$ we define
$$
L(U)=\{v\in V\,|\, f(v,U)=0\}\text{ and }R(U)=\{v\in V\,|\,
f(U,v)=0\}.
$$
Note that $f$ is non-degenerate if $L(V)=0=R(V)$. We set
$Rad(f)=L(V)\cap R(V)$.

\begin{lemma}
\label{lem14} (a) If $\ell=2$ then $L(f)=M(f)$.

(b) If $\ell\neq 2$ then $L(f)\cap M(f)$ is the ideal of $L(f)$ of
all endomorphisms $x$ of $V$ that satisfy $xV\subseteq Rad(f)$.

(c) If $\ell\neq 2$ and $Rad(f)=0$ then $L(f)\cap M(f)=0$.

(d) If $\ell\neq 2$ and $f$ is non-degenerate then $L(f)\cap
M(f)=0$.
\begin{proof} (a) This is obvious.

(b) Let $x\in L(f)\cap M(f)$ and $v,w\in V$. Then
$$
-f(v,xw)=f(xv,w)=f(v,xw)\text{ and }-f(xv,w)=f(v,xw)=f(xv,w)
$$
so
$$
2f(v,xw)=0\text{ and }2f(xv,w)=0.
$$
If $\ell\neq 2$ then $xV\subseteq L(V)\cap R(V)=Rad(f)$. As the
intersection of $L(f)$-submodules of $\gl(V)$, we see that
$L(f)\cap M(f)$ is an ideal of $L(f)$. Moreover, any endomorphism
$x$ of $V$ that sends $V$ to $Rad(f)$ is clearly in $L(f)\cap
M(f)$.

(c) This follows from (b).

(d) This follows from (c).
\end{proof}
\end{lemma}

\begin{lemma}
\label{lem15} If $f$ is non-degenerate and $\ell\neq 2$ then
$M(f)\subseteq L(f)^\perp$.
\begin{proof} Let ${\mathcal B}$ be a basis of $V$ and let $A$ be the
Gram matrix of $f$ relative to ${\mathcal B}$. Let $x\in L(f)$ and
$y\in M(f)$ have respective matrices $X,Y\in\gl(m)$  relative to
${\mathcal B}$. Then
$$
X'A=-AX,\quad Y'A=AY.
$$
Since $f$ is non-degenerate, $A$ is invertible, whence
\begin{equation}
\label{traq} X=-A^{-1}X'A,\quad Y=A^{-1}Y'A.
\end{equation}
It follows that
$$
XY=(-A^{-1}X'A)(A^{-1}Y'A)=-A^{-1}X'Y'A.
$$
Taking traces yields
$$
\tr(XY)=-\tr(X'Y')=-\tr(Y'X')=-\tr((XY)')=-\tr(XY).
$$
Therefore $2\tr(XY)=0$. Since $\ell\neq 2$, we infer $\tr(XY)=0$.
\end{proof}
\end{lemma}

Suppose that $f$ is non-degenerate. Then given, $x\in\gl(V)$,
there exists a unique $x^*\in\gl(V)$, the adjoint of $x$,
satisfying
$$
f(xv,w)=f(v,x^*w),\quad v,w\in V.
$$
In matrix terms, if ${\mathcal B}$ is a basis of $V$ and $A,X,X^*$
are the matrices of $f,x,x^*$, then
$$
X'A=AX^*,
$$
which has the unique solution
$$
X^*=A^{-1}X'A.
$$
Observe that $x\in L(f)\Leftrightarrow x^*=-x$ and $y\in
M(f)\Leftrightarrow y^*=y$.

\begin{lemma}
\label{lem16} Suppose $f$ is non-degenerate as well as symmetric
or skew-symmetric. Then
$$
x^{**}=x,\quad x\in\gl(V).
$$
\begin{proof} Let $v,w\in V$. Then
$$
f(xv,w)=f(v,x^*w)=\pm f(x^*w,v)=\pm f(w,x^{**}v)=(\pm 1)^2f(x^{**}
v,w).
$$
\end{proof}
\end{lemma}

\begin{lemma}
\label{lem17} Suppose $f$ is non-degenerate, symmetric or
skew-symmetric, and $\ell\neq 2$. Then
$$
\gl(V)=L(f)\oplus M(f).
$$
\begin{proof} Given $z\in\gl(V)$ let $x=(z-z^*)/2$ and
$y=(z+z^*)/2$. Then $z=x+y$. Moreover, by Lemma \ref{lem16}, $x\in
L(f)$ and $y\in M(f)$. Furthermore, $L(f)\cap M(f)=0$ by Lemma
\ref{lem14}, as required.
\end{proof}
\end{lemma}

\begin{cor}
\label{cor12} Suppose $f$ is non-degenerate, symmetric or
skew-symmetric, and $\ell\neq 2$. Then
$$
\gl(V)=L(f)\perp M(f).
$$
\begin{proof} Since $\varphi$ is non-degenerate, this follows from Lemmas
\ref{lem15} and \ref{lem17}.
\end{proof}
\end{cor}

\begin{lemma}
\label{lem18} Suppose $f$ is non-degenerate and $\ell\neq 2$. Then
$L(f)\subseteq \sl(V)$.
\begin{proof} By definition $\s\subseteq M(f)$ and by Lemma
\ref{lem15} we have $M(f)\subseteq L(f)^\perp$. Hence $\s\subseteq
L(f)^\perp$. Since $\vp$ is non-degenerate, this yields
$$
L(f)=L(f)^{\perp\perp}\subseteq \s^\perp=\sl(V).
$$
Alternatively, let $x\in L(f)$ and take traces in (\ref{traq}) to
get $\tr(x)=-\tr(x)$.
\end{proof}
\end{lemma}

Since $\sl(V)$ is an ideal of $\gl(V)$ it follows from Lemma
\ref{lem11} that $M(f)\cap \sl(V)$ is an $L(f)$-submodule of
$\gl(V)$.

\begin{cor}
\label{cor13} Suppose $f$ is non-degenerate, symmetric or
skew-symmetric, and $\ell\nmid 2m$. Then $\gl(V)$ has the
following decomposition into perpendicular $L(f)$-submodules:
$$
\gl(V)=L(f)\perp (M(f)\cap \sl(V))\perp\s.
$$
\begin{proof} If we had $M(f)\subseteq \sl(V)$ then $\ell\neq 2$ together with Lemmas \ref{lem17} and \ref{lem18} would imply $\gl(V)\subseteq
\sl(V)$, which is impossible. It follows that $M(f)\cap \sl(V)$ is
a hyperplane of $M(f)$. Since $\ell\nmid m$, then $\s\cap M(f)\cap
\sl(V)=0$, with $\s$ and $M(f)\cap \sl(V)$ perpendicular to each
other, so $M(f)=(M(f)\cap \sl(V))\perp\s$. Replacing this in
Corollary \ref{cor12} yields the desired result.
\end{proof}
\end{cor}

\section{Multiple identifications}\label{secid}

Let $U$ and $W$ be $L$-modules for a Lie algebra $L$. Then $U\otimes W$ becomes an $L$-module via
$$
x\cdot (u\otimes v)=x\cdot u\otimes v+u\otimes x\cdot v.
$$
We may view the symmetric and exterior squares $S^2(U)$ and
$\Lambda^2(U)$ as $L$-submodules of $U\otimes U$. If $\ell=2$ then
$\Lambda^2(U)\subseteq S^2(U)$, while if $\ell\neq 2$ then
$U\otimes U=\Lambda^2(U)\oplus S^2(U)$.

Suppose $f$ is non-degenerate. Since $f$ is $L(f)$-invariant, the
map $\Gamma_1:V\to V^*$ induced by $f$, namely
$$
\Gamma_1(v)=f(v,-),
$$
is an isomorphism of $L(f)$-modules. This, in turn, yields the
$L(f)$-isomorphism $\Gamma_1\otimes 1_V:V\otimes V\to V^*\otimes
V$. On the other hand, we have the natural $\gl(V)$-isomorphism
$\Gamma_2:V^*\otimes V\to\gl(V)$, given by
$$
\Gamma_2(\delta\otimes w)(u)=\delta(u)w,\quad \delta\in
V^*,\,u,w\in V.
$$
It follows that $\Gamma=\Gamma_2\circ(\Gamma_1\otimes 1_V)$ is an
$L(f)$-isomorphism $V\otimes V\to\gl(V)$, given by
$$
\Gamma(v\otimes w)(u)=f(v,u)w,\quad u,v,w\in V.
$$

We wonder what are $L(f)$-submodules of $V\otimes V$ corresponding
to $L(f)$ and $M(f)$ under~$\Gamma$. If $f$ is symmetric or
skew-symmetric, the answer is as follows.

\begin{prop}
\label{gam} Suppose that $f$ is non-degenerate.

(a) If $\ell\neq 2$ and $f$ is skew-symmetric, then $\Gamma$ sends
$\Lambda^2(V)$ onto $M(f)$ and $S^2(V)$ onto $L(f)$.

(b) If $\ell\neq 2$ and $f$ is symmetric, then $\Gamma$ sends
$\Lambda^2(V)$ onto $L(f)$ and $S^2(V)$ onto $M(f)$.

(c) If $\ell=2$ and $f$ is symmetric then $\Gamma$ sends $S^2(V)$
onto $L(f)=M(f)$.
\end{prop}

\begin{proof} The dimensions of $L(f)$ and $M(f)$ we computed in
\ref{subcla}. On the other hand, it is well-known that
$$
\dim S^2(V)={{m+1}\choose{2}}\text{ and  }\dim
\Lambda^2(V)={{m}\choose{2}}.
$$
Furthermore, given $u_1,u_2,v,w\in V$, we have
$$
f(\Gamma(v\otimes w)u_1,u_2)=f(v,u_1)f(w,u_2),
$$
$$
f(u_1,\Gamma(v\otimes w)u_2)=f(u_1,w)f(v,u_2).
$$
Since $S^2(V)$ is spanned by all $v\otimes v$ and $v\otimes
w+w\otimes v$, and $\Lambda^2(V)$ by all $v\otimes w-w\otimes v$,
the above information combines to yield the desired result.
\end{proof}

\begin{theorem}\label{matver} Suppose that $\ell\neq 2$, that $m=2n$, and that $f$ is non-degenerate and skew-symmetric.
Then the following $L(f)$-modules are isomorphic:

(1) $M(f)=\{y\in\gl(V)\,|\, f(yv,w)=f(v,yw)\text{ for all }v,w\in
V\}$.

(2) $\Lambda^2(V)$, the second exterior power of the natural $L(f)$-module $V$.

(3) $L(f)^\perp$, the orthogonal complement of $L(f)$ relative to the bilinear form (\ref{tra}).

(4) The space of all
\begin{equation}
\label{qa}
\left(\begin{array}{cc}
    a & b \\
   c & a'
  \end{array}
\right)\in\gl(2n)
\end{equation}
such that $a\in\gl(n)$ and $b,c\in\gl(n)$ are skew-symmetric.
\end{theorem}

\begin{proof} We know from Proposition \ref{gam} that $M(f)\cong \Lambda^2(V)$, while Corollary \ref{cor12} shows that $M(f)=L(f)^\perp$ relative to $\vp$.
The matrix description of $M(f)$ is taken from \S\ref{subcla}.
\end{proof}

In a similar manner, we derive the following result.

\begin{theorem}\label{matver2} Suppose that $\ell\neq 2$, $f$ is non-degenerate and symmetric.
Then the following $L(f)$-modules are isomorphic:

(1) $M(f)=\{y\in\gl(V)\,|\, f(yv,w)=f(v,yw)\text{ for all }v,w\in
V\}$.

(2) $S^2(V)$, the second symmetric power of the natural
$L(f)$-module $V$.

(3) $L(f)^\perp$, the orthogonal complement of $L(f)$ relative to the bilinear form (\ref{tra}).

(4) The space of all $A\in\gl(m)$ satisfying
$d_iA_{ij}-d_jA_{ji}=0$ for all $1\leq i,j\leq m$, where
$D=\mathrm{diag}(d_1,\dots,d_m)$ is the Gram matrix of $f$
relative to a basis of $V$.

In particular, if $F=F^2$ then $M(f)$ is isomorphic to the
$L(f)$-module of all symmetric $m\times m$ matrices.
\end{theorem}



\begin{lemma}\label{contra} The contraction map $\Omega:V\otimes V\to F$, given by $v\otimes w\mapsto f(v,w)$, is an $L(f)$-homomorphism,
which is surjective if and only if $f\neq 0$.
\end{lemma}

\begin{proof} This is an easy calculation.
\end{proof}

\begin{lemma}\label{contra2} Suppose that $f$ is non-degenerate. Then $\Omega$ can be identified with the trace map $\tr:\gl(V)\to F$,
in the sense that $\Omega(v\otimes w)=\tr(\Gamma(v\otimes w))$ for all $v,w\in V$.
\end{lemma}

\begin{proof} Let $v_1,\dots,v_m$ be a basis of $V$ and let $w_1,\dots,w_m$ be the dual basis of $V$ relative to $f$, i.e.,
such that $f(v_i,w_j)=\delta_{ij}$. Then $\Gamma(v_i\otimes
w_j)(w_k)=\delta_{ik}w_j$, so
$$
\tr(\Gamma(v_i\otimes
w_j))=\delta_{ij}=f(v_i,w_j)=\Omega(v_i\otimes w_j),
$$
which implies the result for all $v,w\in V$.
\end{proof}

\begin{cor}\label{loro} Suppose that $\ell\neq 2$ and $f$ is non-degenerate.

(a) If $f$ is skew-symmetric then $\Lambda^2(V)$ is not contained in the the kernel of $\Omega$, so
$\Lambda^2(V)\cap \ker\Omega$ is an $L(f)$-submodule of $\Lambda^2(V)$ of codimension 1, which
corresponds to $M(f)\cap\sl(V)$ under $\Gamma$, and in matrix form to all matrices (\ref{qa}) with $a\in\sl(n)$.

(b) If $f$ is symmetric then $S^2(V)$ is not contained in the the
kernel of $\Omega$, so $S^2(V)\cap \ker\Omega$ is an
$L(f)$-submodule of $S^2(V)$ of codimension 1, which corresponds
to $M(f)\cap\sl(V)$ under $\Gamma$, and in matrix form to all
$A\in\gl(m)$ as described in Theorem \ref{matver2} satisfying
$\tr(A)=0$.

(c) If $f$ is symmetric or skew-symmetric then $\s$ is contained in $M(f)\cap\sl(V)$ if and only if $\ell|m$.
\end{cor}


\begin{lemma}\label{quot} The map $V\otimes V\to \Lambda^2(V)$ given by
$$
v\otimes w\mapsto  v\otimes w-w\otimes v
$$
is an epimorphism of $\gl(V)$-modules with kernel $S^2(V)$.
\end{lemma}

\begin{proof} This is clear.
\end{proof}

\begin{theorem}\label{lami} Suppose that $\ell=2$ and $f$ is non-degenerate and
symmetric. Then

(1) $\Gamma$ maps the $L(f)$-submodule $\Lambda^2(V)$ of $S^2(V)$
onto $L(f)^{(1)}=[L(f),L(f)]$, an ideal of $L(f)$ of
codimension~$m$.

(2) $\gl(V)/L(f)\cong L(f)^{(1)}$ as $L(f)$-modules.

(3) Suppose $f$ is alternating. Then $m=2n$ and there is a basis
${\mathcal B}$ of $V$ relative to which $f$ has Gram matrix $J$,
as defined in (\ref{defjo}).

Moreover, $L(J)$ (resp. $L(J)^{(1)}$) consists of all matrices
\begin{equation}
\label{deri}
\left(%
\begin{array}{cc}
  A & B \\
  C & A' \\
\end{array}%
\right)
\end{equation}
such that $A,B,C\in\gl(n)$ and $B,C$ are symmetric (resp.
alternating).

Furthermore, $L(J)^{(2)}$ consists of all matrices (\ref{deri})
such that $B,C$ are alternating and $\tr(A)=0$. In particular,
$L(J)^{(2)}$ has codimension 1 in $L(J)^{(1)}$.

(4) Suppose $f$ is alternating and let ${\mathcal B}$ and $J$ be
as above. Let $\Delta:\Lambda^2(V)\to F$ be the $L(f)$-epimorphism
given by
$$
v\wedge w=v\otimes w+w\otimes v\mapsto f(v,w).
$$
Then $L(J)^{(2)}$ corresponds to $\ker\Delta$ under the
$L(f)$-isomorphism $M_{\mathcal B}\circ \Gamma$.

(5) Suppose $f$ is non-alternating. Then there is a basis
${\mathcal B}$ of $V$ relative to which $f$ has Gram matrix
$$
D=\mathrm{diag}(d_1,\dots,d_m),\quad 0\neq d_i\in F.
$$
Moreover, $L(D)$ (resp. $L(D)^{(1)}$) consists of all $A\in\gl(m)$
such that $d_iA_{ij}=d_j A_{ji}$ (resp. $d_iA_{ij}=d_j A_{ji}$ and
$A_{ii}=0$). Furthermore, $L(f)^{(1)}$ is perfect provided~$m\neq
2$. In particular, if $D=I_m$, which occurs when $F=F^2$, then
$L(D)$ (resp. $L(D)^{(1)}$) consists of all symmetric (resp.
alternating) matrices in $gl(m)$.

(6) The space $\s$ is contained in $L(f)^{(1)}$ if and only if $f$
is alternating, in which case $\s$ is contained in $L(f)^{(2)}$ if
and only if $4|m$.
\end{theorem}

\begin{proof} By Proposition \ref{gam}, the $L(f)$-isomorphism
$\Gamma:V\otimes V\to\gl(V)$ sends $S^2(V)$ onto $L(f)$.
Therefore,
$$
\Gamma(L(f)\cdot S^2(V))=L(f)\cdot L(f)=L(f)^{(1)}.
$$
One the other hand, we easily verify that
$$
L(f)\cdot S^2(V)\subseteq \Lambda^2(V),
$$
which yields
\begin{equation}
\label{copn} \Gamma(\Lambda^2(V))\supseteq L(f)^{(1)}.
\end{equation}
Since $\Lambda^2(V)$ has codimension $m$ in $S^2(V)$, in order to
show that equality prevails in (\ref{copn}) it suffices to show
that $L(f)^{(1)}$ has codimension $m$ in $L(f)$.

Suppose first that $f$ is non-alternating. By \cite{K}, Theorem
20, there is a basis ${\mathcal B}$ of $V$ relative to which $f$
has Gram matrix $D=\mathrm{diag}(d_1,\dots,d_m)$, where $0\neq
d_i\in F$. We easily see that $L(D)$ consists of all $A\in\gl(m)$
such that $d_iA_{ij}=d_j A_{ji}$, with the following
multiplication table:
$$
[e_{ii},d_je_{ij}+d_i e_{ji}]=d_je_{ij}+d_i e_{ji},
$$
$$
[d_je_{ij}+d_i e_{ji},d_ke_{ik}+d_i e_{ki}]=d_i(d_k e_{jk}+d_j
e_{kj}).
$$
This proves (5) and completes the verification of (\ref{copn})
when $f$ is non-alternating.

Suppose next that $f$ is alternating. By \cite{K}, Theorem 19,
$m=2n$ and there is a basis ${\mathcal B}$ of $V$ relative to
which $f$ has Gram matrix $J$, as defined in
(\ref{defjo}). As seen in \S\ref{subtra}, $L(J)$ consists of
the stated matrices. Moreover, we have
\begin{equation}
\label{req1} [\gl(n),\gl(n)]=\sl(n),
\end{equation}
\begin{equation}
\label{req2} \left[\left(%
\begin{array}{cc}
  0 & B \\
  0 & 0 \\
\end{array}%
\right),\left(%
\begin{array}{cc}
  0 & 0 \\
  C & 0 \\
\end{array}%
\right)\right]=\left(%
\begin{array}{cc}
  BC & 0 \\
  0 & CB \\
\end{array}%
\right),
\end{equation}
\begin{equation}
\label{req3} \left[\left(%
\begin{array}{cc}
  A & 0 \\
  0 & A' \\
\end{array}%
\right),\left(%
\begin{array}{cc}
  0 & B \\
  0 & 0 \\
\end{array}%
\right)\right]=\left(%
\begin{array}{cc}
  0 & AB+(AB)' \\
  0 & 0 \\
\end{array}%
\right),
\end{equation}
\begin{equation}
\label{req4} \left[\left(%
\begin{array}{cc}
  A & 0 \\
  0 & A' \\
\end{array}%
\right),\left(%
\begin{array}{cc}
  0 & 0 \\
  C & 0 \\
\end{array}%
\right)\right]=\left(%
\begin{array}{cc}
  0 & 0 \\
  CA+(CA)' & 0 \\
\end{array}%
\right).
\end{equation}

Combining (\ref{req1}) and (\ref{req2}) with $B=I_n$ and
$C=\mathrm{diag}(1,0,\dots,0)$ we deduce
$$
\left(%
\begin{array}{cc}
  A & 0 \\
  0 & A' \\
\end{array}%
\right)\in L(J)^{(1)},\quad A\in\gl(n).
$$
As seen in \S\ref{subtra}, the map $\gl(n)\to\mathrm{Alt}(n)$, given by $A\mapsto A+A'$, is surjective.
This, together with
(\ref{req3}) and (\ref{req4}) applied to the special case
$B=I_n=C$, imply that for all alternating matrices $B,C\in\gl(n)$,
we have
$$
\left(%
\begin{array}{cc}
  0 & B \\
  0 & 0 \\
\end{array}%
\right),\left(%
\begin{array}{cc}
  0 & 0 \\
  C & 0 \\
\end{array}%
\right)\in L(J)^{(1)}.
$$
It now follows from (\ref{req1})-(\ref{req4}) that $L(J)^{(1)}$
consists of all matrices (\ref{deri}) such that $B,C$ are
alternating. This proves the first two statements of (3) and
completes the proof of (1).

Using (\ref{req3}) and (\ref{req4}) with $A=e_{ii}$ and
$B=e_{ij}+e_{ji}=C$, where $i\neq j$, we see that for all
alternating matrices $B,C\in\gl(n)$, we have
$$
\left(%
\begin{array}{cc}
  0 & B \\
  0 & 0 \\
\end{array}%
\right),\left(%
\begin{array}{cc}
  0 & 0 \\
  C & 0 \\
\end{array}%
\right)\in L(J)^{(2)}.
$$
Let $B,C\in\gl(n)$ be alternating. We infer from (\ref{skew}) that
$$\tr(BC)=0.$$
It now follows from (\ref{req1})-(\ref{req4}) that $L(J)^{(2)}$
consists of all matrices (\ref{deri}) such that $B,C$ are
alternating and $\tr(A)=0$, which completes the proof of (3).

Suppose still that $f$ is alternating. Then $f$ induces a linear
map $\Lambda^2(V)\to F$, namely $\Delta$. It is clear that
$\Delta$ is an $L(f)$-epimorphism. Let
$${\mathcal B}=\{v_1,\dots,v_n,w_1,\dots,w_n\}$$ be the given
basis $V$, relative to which $f$ has Gram matrix $J$. Then
$M_{\mathcal B}\circ\Gamma$ satisfies:
$$
v_i\wedge v_j\mapsto \left(%
\begin{array}{cc}
  0 & e_{ij}+e_{ji} \\
  0 & 0 \\
\end{array}%
\right), w_i\wedge w_j\mapsto \left(%
\begin{array}{cc}
  0 & 0 \\
  e_{ij}+e_{ji} & 0 \\
\end{array}%
\right), v_i\wedge w_j\mapsto \left(%
\begin{array}{cc}
  e_{ij} & 0 \\
  0 & e_{ji} \\
\end{array}%
\right).
$$
This, (1) and (3) show that $M_{\mathcal B}\circ\Gamma$ sends
$\ker\Delta$ onto $L(J)^{(2)}$, which is (4).

Now (3) and (5) yield the first part of (6), while (3) gives the
second. Finally, we may apply (1), Proposition \ref{gam} and Lemma
\ref{quot} to derive (2).
\end{proof}

\begin{note}\label{errorb}{\rm Suppose that $\ell=2$ and $f$ is non-degenerate and alternating.
There are two mistakes in \cite{B}, Chapter I, \S 6, Exercise
25(b). In their notation, it is claimed that $\mathfrak
b=L(f)^{(1)}$, when in fact $\mathfrak a=L(f)^{(1)}$ and
$\mathfrak b=L(f)^{(2)}$. This holds for any even $m$, while their
claim was made for $m\geq 6$. It is also claimed that if $m\geq 6$
then $\mathfrak {b/c}$ is simple, when in fact $\mathfrak c=\s$ is
only included in $\mathfrak b=L(f)^{(2)}$ when $4|m$. }
\end{note}

\section{The $L(f)$-module $L(f)^{(1)}$ when $f$ is symmetric and non-alternating}
\label{secSoSimple}

We assume throughout this section that $m\ge 3$ and let
$L=L(I_m)=\so(m)$.

On the one hand, if $\ell\ne 2$ then $L$ consists of all skew-symmetric matrices and $L^{(1)}=L$. On the other hand, when $\ell=2$, $L$ is the set of all symmetric matrices and $L^{(1)}$ consists of all alternating
matrices. In any characteristic, the derived algebra $L^{(1)}$ is
spanned by the matrices $E_{ij}$ with $i<j$, where $E_{ij}$ is
defined as $e_{ij}-e_{ji}$ for any $i,j$. The following
multiplication rules can be easily verified.
\begin{itemize}
\item $[E_{ij},E_{jk}]=E_{ik}$ for $i\ne j$ and $j\ne k$; \item
$[E_{ij},E_{rs}]=0$ if $i,j,r,s$ are all different to each other.
\end{itemize}

\begin{theorem}
\label{sosimple1} If $m=3$ or $m\ge 5$ then $L^{(1)}$ is a simple
Lie algebra. If $m\ge 3$ and $\ell=2$ then $L^{(1)}$ is an
irreducible $L$-module.
\end{theorem}

\begin{proof}
When $m=3$, $L^{(1)}$ has basis $\{E_{12},E_{23},E_{31}\}$ with
multiplication table given by $[E_{ij},E_{jk}]=E_{ik}$ for
$\{i,j,k\}=\{1,2,3\}$. So $L^{(1)}$ is 3-dimensional and perfect,
and therefore   simple. Assume henceforth that $m\ge 4$.

Suppose first that $I$ is an ideal of $L^{(1)}$ such that
$E_{ij}\in I$ for some $i\ne j$. Let $r\ne s$ be indices such that
$\{i,j\}\ne \{r,s\}$. If $\{i,j\}\cap \{r,s\}=\emptyset$ or $j=r$,
then $$ [E_{ri},[E_{ij},E_{js}]]=E_{rs}\in I.
$$
So $I=L^{(1)}$ if $I$ contains a basis element.

Now suppose $m\ge 5$ and let $I$ be a nonzero ideal of $L^{(1)}$.
Let $x\in I$ with $x_{ij}\ne 0$ for some $i\ne j$. As $m\ge 5$, we
can pick indices $r,s,t$ such that $|\{i,j,r,s,t\}|=5$. Since
$$
(\mathrm{ad}E_{ts}\circ\mathrm{ad}E_{rs}\circ\mathrm{ad}E_{ji}
\circ\mathrm{ad}E_{rj})(x) = x_{ij}E_{tj},
$$
we deduce $E_{tj}\in I$. Thus $I=L^{(1)}$.

Finally suppose $m\ge 4$ and $\ell=2$, and let $W$ be a nonzero
$L$-submodule of $L^{(1)}$. Let $x\in W$ with $x_{ij}\ne 0$ for
some $i\ne j$. Let $r,s$ be indices such that $|\{i,j,r,s\}|=4$.
As
$$
(\mathrm{ad}E_{rs}\circ \mathrm{ad}E_{jr}\circ
\mathrm{ad}e_{ii})(x) = x_{ij}E_{is},
$$
we have $E_{is}\in W$. Since $W$ is in particular an ideal of
$L^{(1)}$, we infer $W=L^{(1)}$.
\end{proof}

\section{A composition series of the $\so(m)$-module $\gl(m)$ when $\ell\neq 2$}\label{dondeso}

We suppose throughout this section that $\ell\neq 2$, that $m\geq
2$, and that $f$ is non-degenerate and symmetric. We further assume that $L=L(I_m)=\so(m)$.

Note that $L$ consists of all skew-symmetric matrices and is
spanned by the matrices $E_{ij}$ with $i<j$, as defined in
\S\ref{secSoSimple}.

\begin{prop}\label{sosimple}
If $m=3$ or $m\ge 5$ then $L(f)$ is a simple Lie algebra.
\end{prop}

\begin{proof}
This follows from Theorem \ref{sosimple1}, extending scalars if
necessary.
\end{proof}

\begin{note}\label{noso1}{\rm Suppose that $m=4$.
Then $L(f)$ is 6-dimensional. Moreover, $L(f)$ is a simple Lie
algebra if the discriminant of $f$ is not a square in $F$, and the
direct sum of two 3-dimensional simple ideals otherwise.

This can be found in  \cite{B}, Chapter I, \S 6, Exercise 26(b).
An alternative approach via current Lie algebras, independent of
whether $\ell\neq 2$ or not, can be found in \cite{CS}. }
\end{note}

\begin{note}\label{noso2}{\rm If $m=2$ then $L(f)$ is 1-dimensional.}
\end{note}

Consider next the $L$-module $M=M(I_m)$ consisting of all
symmetric matrices. This module has basis
$\{A_{ij}:i\le j\}$, where $A_{ij}$ is defined as $e_{ij}+e_{ji}$
for all $i,j$. The matrices $E_{rs}$ act on the $A_{ij}$ according
to the following rules.
\begin{itemize}
\item $[E_{ij},A_{ij}]=A_{ii}-A_{jj}$; \item
$[E_{ij},A_{jj}]=2A_{ij}$ for $i\ne j$; \item
$[E_{ij},A_{jk}]=A_{ik}$ if $\{i,j,k\}$ has size 3; \item
$[E_{ij},A_{rs}]=0$ if $\{i,j,r,s\}$ has size 4.
\end{itemize}

\begin{theorem}\label{pus0}
Suppose $m\ge 4$. Let $M^0 = M(I_m)\cap \mathfrak{sl}(m)$. Then:

(1) If $\ell\nmid m$, then $M^0$ is an irreducible $L$-module.

(2) If $\ell|m$, then $\mathfrak{s}$ is the only non-trivial
$L$-submodule of $M^0$, so $M^0/\mathfrak{s}$ is an irreducible
$L$-submodule.
\end{theorem}

\begin{proof}
Let $W\ne 0$ be an $L$-submodule of $M^0$. Suppose first that $W$
consists only of diagonal matrices and let $0\ne h\in W$. If
$i\ne j$ then $[E_{ij},h]=(h_{jj}-h_{ii})A_{ij}\in W$, whence
$h_{ii}=h_{jj}$, i.e., $h$ is scalar. This implies
$W=\mathfrak{s}$ in the case $\ell|m$, and is a contradiction in
the case $\ell\nmid m$.

Now suppose $A_{ij}\in W$ for some $i\ne j$. Let $r,s$ be distinct indices
such that $\{i,j\}\ne \{r,s\}$. If $j=r$ then
$[E_{si},A_{ij}]=A_{sr}\in W$, whereas $\{i,j\}\cap
\{r,s\}=\emptyset$ implies $[E_{sj},[E_{ri},A_{ij}]]=A_{sr}\in W$.
In addition $[E_{sr},A_{sr}]=A_{ss}-A_{rr}\in W$. Therefore
$W=M^0$ if $W$ contains a basis element $A_{ij}$ with $i\ne j$.

Finally, suppose that $W$ contains a non-diagonal matrix $x$. So
$x_{ij}\ne 0$ for some $i\ne j$. As $m\ge 4$, we can find indices
$r,s$ such that $\{i,j,r,s\}$ has size 4. Since
$$
(\mathrm{ad}E_{ir}\circ \mathrm{ad}E_{rs}\circ
\mathrm{ad}E_{ir}\circ \mathrm{ad}E_{ij}\circ
\mathrm{ad}E_{rs}\circ \mathrm{ad}E_{ir})(x) = -2x_{ij}A_{ir},
$$
we deduce $A_{ir}\in W$, which implies $W=M^0$.
\end{proof}

\begin{cor}
\label{pus2} Suppose that $m\geq 4$. Then

(1) If $\ell$ does not divide $m$ then $M(f)\cap\sl(V)$ is an
irreducible $L(f)$-module of dimension~${{m+1}\choose{2}}-1$.

(2) If $\ell$ divides $m$ then $M(f)\cap\sl(V)/\s$ is an
irreducible $L(f)$-module of dimension ${{m+1}\choose{2}}-2$.
\end{cor}

\begin{proof}
The stated dimensions are clear from the matrix version of $M(f)$.
Irreducibility follows from Theorem \ref{pus0}, extending scalars
if necessary.
\end{proof}

\begin{note}{\rm Suppose that $\ell\nmid m$ and $m>2$. Then Corollaries \ref{loro} and
\ref{pus2} show that the traceless matrices $A\in\gl(m)$ described
in part (4) of Theorem \ref{matver2} form an irreducible
$L(f)$-module of dimension ${{m+1}\choose{2}}-1$, isomorphic to
the kernel of the contraction map $S^2(V)\to F$ given by $v\otimes
w+w\otimes v\mapsto f(v,w)$.

When $F=\C$ and $m>4$ this gives an elementary matrix description
of $V(2\lambda_1)$, where $\lambda_1$ is first fundamental module
of the orthogonal Lie algebra $\so(m)$, as the space of all
traceless $m\times m$ symmetric matrices.}
\end{note}

\begin{note}\label{som22}{\rm
Suppose that $m=3$. If $\ell\ne 3$ then $M^0$ is an irreducible
$L$-module. Now suppose $\ell=3$. If $-1$ is not a square in $F$,
then $M^0/\s$ is irreducible. However, if $-1$ is a square in $F$,
say $i^2=-1$, the $L$-submodules of $M^0$ are $X$, $Y$ and $\s =
X\cap Y$, where }
$$
X = \mathrm{span} \left\lbrace I_3, \begin{pmatrix}
0 & 0 & 0 \\
0 & 1 & i \\
0 & i & -1
\end{pmatrix}, \begin{pmatrix}
0  & i & -1 \\
i  & 0 & 0 \\
-1 & 0 & 0
\end{pmatrix} \right\rbrace,
$$
$$
Y = \mathrm{span} \left\lbrace I_3, \begin{pmatrix}
0 & 0  & 0 \\
0 & -1  & i \\
0 & i & 1
\end{pmatrix}, \begin{pmatrix}
0 & i & 1 \\
i & 0 & 0 \\
1 & 0 & 0
\end{pmatrix} \right\rbrace.
$$

\end{note}

\begin{note}\label{som33}{\rm
Suppose $m=2$. If $-1$ is not a square in $F$ then $M^0$ is
irreducible as a module over $L$. If $-1$ is a square in $F$, say
$i^2=-1$, the only $L$-submodules of $M^0$ are $Fx$ and $Fy$,
where
$$
x = \begin{pmatrix}
1 & i \\
i & -1
\end{pmatrix} \quad \text{and} \quad y = \begin{pmatrix}
-1  & i \\
i & 1
\end{pmatrix}.
$$
}
\end{note}

Combining the results of this section with  Corollary \ref{cor12},
Proposition \ref{gam}, Theorem \ref{matver2}, and Corollary
\ref{loro}, we obtain the following theorem.

\begin{theorem} \label{61} Suppose that $\ell\neq 2$, that $m\geq 2$, and that $f$ is
non-degenerate and symmetric. Then

(1) $M(f)$ is the orthogonal complement to $L(f)$ with respect to
the bilinear form $\vp:\gl(V)\times \gl(V)\to F$, given by
$\vp(x,y)=\tr(xy)$. Moreover, there is a basis of $V$ relative to
which $f$ has Gram matrix $D=\mathrm{diag}(d_1,\dots,d_n)$ and,
relative to this basis, $M(f)$ consists of all $A\in\gl(m)$ such
that $d_iA_{ij}=d_j A_{ji}$. Furthermore, $M(f)$ is isomorphic to
$S^2(V)$ as $L(f)$-module.

(2) $M(f)\cap\sl(V)$ consists, relative to the above basis, of all
matrices $A\in\gl(m)$ such that $d_iA_{ij}=d_j A_{ji}$ and
$\tr(A)=0$, and is isomorphic to the kernel of the contraction
$L(f)$-epimorphism $S^2(V)\to F$ given by $vw\to f(v,w)$.

(3) If $m\geq 4$ and $\ell\nmid m$ then $M(f)\cap\sl(V)$ is an
irreducible $L(f)$-module of dimension~${{m+1}\choose{2}}-1$.

(4) If $m\geq 4$ and $\ell\mid m$ then $M(f)\cap\sl(V)/\s$ is an
irreducible $L(f)$-module of dimension ${{m+1}\choose{2}}-2$.

(5) If $m=3$ or $m\geq 5$ then $L(f)$ is a simple Lie algebra,
isomorphic to $\gl(V)/M(f)$ and $\Lambda^2(V)$ as $L(f)$-modules.

(6) The following are composition series of the $L(f)$-module
$\gl(V)$:
$$
0\subset \s\subset M(f)\cap\sl(V)\subset M(f)\subset \gl(V),\text{
if }m\geq 4\text{ and }\ell|m,
$$
$$
0\subset M(f)\cap\sl(V)\subset M(f)\subset \gl(V),\text{ if }m\geq
4\text{ and }\ell\nmid m.
$$
In any case, $M(f)/M(f)\cap\sl(V)$ is the trivial $L(f)$-module.
\end{theorem}

\section{A composition series of the $\sy(2n)$-module $\gl(2n)$ when $\ell\neq 2$}\label{donde}

We assume throughout this section that $\ell\neq 2$, that $m=2n$,
and that $f$ is non-degenerate and skew-symmetric.

\begin{theorem}\label{spsimple} The symplectic Lie algebra $L(f)$ is
simple.
\end{theorem}

\begin{proof}
We show that $L=L(J)$ is simple, where $J$ is defined in
(\ref{defjo}). We can assume $m\ge 4$ because $\sy(2)=\sl(2)$. Let
$I$ be a nonzero ideal of $L$ and suppose $0\ne x \in I$. Write
$$
x = \begin{pmatrix}
a & b \\
c & -a'
\end{pmatrix}.
$$
with $a,b,c\in \gl(n)$ and $b,c$ symmetric. Let
$$
y = \begin{pmatrix}
0 & I_n \\
0 & 0
\end{pmatrix}, \quad z = \begin{pmatrix}
0 & 0 \\
I_n & 0
\end{pmatrix}.
$$

Let $S,B$ be the subspaces of $L$ defined by
$$
S = \left\lbrace \begin{pmatrix}
0 & s\\
0 & 0
\end{pmatrix} : s\in \gl(n), \, s'=s \right\rbrace, \quad
B = \left\lbrace \begin{pmatrix}
0 & 0\\
s & 0
\end{pmatrix} : s\in \gl(n), \, s'=s \right\rbrace.
$$

Since $(\mathrm{ad}y\circ \mathrm{ad}z \circ
\mathrm{ad}z)(x)=-2(b\oplus -b)$ and $(\mathrm{ad}z\circ
\mathrm{ad}y \circ \mathrm{ad}y)(x)=2(c\oplus -c)$, we can assume
that $b=c=0$ and $a\ne 0$. Moreover, we can assume that $a$ is not
scalar, for otherwise $[z,[e_{1,n+2}+e_{2,n+1},a]]$ has the
desired form. Then the action of $\gl(n)$ on $I$ yields
$u\oplus(-u')\in I$ for all $u\in\sl(n)$. Applying $\mathrm{ad}y$
we obtain $I\cap S\ne 0$, hence $S\subset I$ by Theorem
\ref{ugas33}. Analogously $B\subset I$. Since
$[y,e_{n+1,1}]=e_{11}\oplus -e_{11}$, we conclude that $I=L$.
\end{proof}



\begin{theorem}
\label{pus} Suppose that $m>2$.

(1) If $\ell\nmid m$ then $M(f)\cap\sl(V)$ is an irreducible
$L(f)$-module of dimension ${{m}\choose{2}}-1$.

(2) If $\ell\mid m$ then $M(f)\cap\sl(V)/\s$ is an irreducible
$L(f)$-module of dimension ${{m}\choose{2}}-2$.
\end{theorem}

\begin{proof}
Let $J$ like in (\ref{defjo}). If $D$ is an $L(J)$-submodule of
$M(J)\cap \sl(m)$ properly containing $M(J)\cap \sl(m) \cap \s $,
we deduce $D=M(J)\cap \sl(m)$ arguing like in (\ref{spsimple}).
\end{proof}

\begin{note}{\rm Suppose that $\ell\nmid m$ and $m>2$. Then Corollary \ref{loro} and Theorem \ref{pus} show that the
subspace of $\gl(2n)$ of all matrices (\ref{qa}) such that
$b,c\in\gl(n)$ are skew-symmetric and $a\in\sl(n)$ is an
irreducible $L(f)$-module of dimension ${{m}\choose{2}}-1$,
isomorphic to the kernel of the contraction map $\Lambda^2(V)\to
F$, $v\wedge w\mapsto f(v,w)$.

When $F=\C$ this gives an elementary matrix description of
$V(\lambda_2)$, the second fundamental module of the symplectic
Lie algebra $\sy(2n)$. }
\end{note}

\begin{note}{\rm If $m=2$ then $M(f)=\s$ is 1-dimensional.}
\end{note}

Combining the results of this section with Corollary \ref{cor12},
Theorems \ref{gam} and \ref{matver}, and Corollary \ref{loro}, we
obtain the following theorem.

\begin{theorem} \label{51} Suppose that $\ell\neq 2$, that $m=2n$, and that $f$ is
non-degenerate and skew-symmetric. Then

(1) $M(f)$ is the orthogonal complement to $L(f)$ with respect to
the bilinear form $\vp:\gl(V)\times \gl(V)\to F$, given by
$\vp(x,y)=\tr(xy)$. Moreover, $M(f)$ consists, relative to
suitable basis of $V$, of all matrices
\begin{equation}\label{maint33}
\left(\begin{array}{cc}
   A & B\\
   C & A'
  \end{array}
\right),\, A,B,C\in\gl(n),\text{ where }B,C\text{ are
skew-symmetric}.
\end{equation}
Furthermore, $M(f)$ is isomorphic to $\Lambda^2(V)$ as
$L(f)$-module.

(2) $M(f)\cap\sl(V)$ consists of all matrices (\ref{maint33}) such
that $\tr(A)=0$ and is isomorphic to the kernel of the contraction
$L(f)$-epimorphism $\Lambda^2(V)\to F$ given by $v\wedge w\to
f(v,w)$.

(3) If $m>2$ and $\ell\nmid m$ then $M(f)\cap\sl(V)$ is an
irreducible $L(f)$-module of dimension ${{m}\choose{2}}-1$.

(4) If $m>2$ and $\ell\mid m$ then $M(f)\cap\sl(V)/\s$ is an
irreducible $L(f)$-module of dimension ${{m}\choose{2}}-2$.

(5) $L(f)$ is a simple Lie algebra, isomorphic to $\gl(V)/M(f)$
and $S^2(V)$ as $L(f)$-modules.

(6) The following are composition series of the $L(f)$-module
$\gl(V)$:
$$
0\subset \s\subset M(f)\cap\sl(V)\subset M(f)\subset \gl(V),\text{
if }m>2\text{ and }\ell|m,
$$
$$
0\subset M(f)\cap\sl(V)\subset M(f)\subset \gl(V),\text{ if
}m>2\text{ and }\ell\nmid m,
$$
$$
0\subset M(f)\subset \gl(V),\text{ if }m=2.
$$
In any case, $M(f)/M(f)\cap\sl(V)$ is the trivial $L(f)$-module.
\end{theorem}

\section{A composition series for the $\so(m)$-module $\gl(m)$ when $\ell=2$}

We assume throughout this section that $\ell=2$, that $m\geq 2$,
and that $f$ is non-degenerate, symmetric and non-alternating.

\begin{theorem} Suppose that $m=3$ or
$m\geq 5$. Then $L(f)^{(1)}$ is a simple Lie algebra, and hence an
irreducible $L(f)$-module, of dimension ${{m}\choose{2}}$.
\end{theorem}

\begin{proof}
This was already proven in \S\ref{dondeso}.
\end{proof}

\begin{prop} Suppose that $m=4$ and let $D$ be the discriminant
of $f$ relative to a basis of $V$. Then

(1) $L(f)^{(1)}$ is 6-dimensional perfect Lie algebra.

(2) If $D\notin F^2$ then $L(f)^{(1)}$ is a simple Lie algebra.

(3) If $D\in F^2$ then $L(f)^{(1)}=S\ltimes R$, where $S$ is a
simple 3-dimensional subalgebra of $L(f)^{(1)}$, and $R$ is
abelian, the solvable radical of $L(f)^{(1)}$ and an irreducible
 $L(f)^{(1)}$-module.

(4) $L(f)^{(1)}$ is an irreducible $L(f)$-module.
\end{prop}

\begin{proof} (1) Since $L(f)$ is 10-dimensional, Theorem \ref{lami} ensures that $L(f)^{(1)}$
is a 6-dimensional perfect Lie algebra.

(2) This can be found in \cite{B}, Chapter I, \S 6, Exercise 26(b)
as well as in \cite{CS}.

(3) This can be found in \cite{CS}.

(4) It suffices to prove this when $F$ is algebraically closed.
Although the result follows from Theorem \ref{sosimple1}, we
provide here an alternative argument. As seen in Theorem
\ref{lami}, $f$ admits $I_4$ as Gram matrix and $L=L(I_m)$ (resp.
$M=L(I_m)^{(1)}$) consists of all symmetric (resp. alternating)
matrices. Thus a basis of $L$ is formed by all $e_{ii}$ and all
$e_{ij}+e_{ji}$, $1\leq i\neq j\leq 4$, and the latter form a
basis of $M$. Set
$$
f_1=e_{12}+e_{21},\; f_2=e_{23}+e_{32},\; f_3=e_{13}+e_{31},
$$
$$
h_1=e_{34}+e_{43},\; h_2=e_{14}+e_{41},\; h_3=e_{42}+e_{24},
$$
$$
g_1=f_1+h_1,\; g_2=f_2+h_2,\; g_3=f_3+h_3,
$$
$$
S=\langle f_1,f_2,f_3\rangle,\; R=\langle g_1,g_2,g_3\rangle.
$$
Then $M=S\ltimes R$, where $S$ is a simple Lie algebra, and $R$ is
an abelian ideal of $M$ and an irreducible $S$-module (isomorphic
to the adjoint module of $S$). It follows that $S$ is the only
non-zero proper $M$-submodule of $M$. Since
$$
[e_{11},g_1]=f_1\notin R,
$$
$M$ is irreducible as $L$-module.
\end{proof}

\begin{note}{\rm If $m=2$ then $L=L(f)$ is solvable of class
2, but not nilpotent, with
$$
\dim L=3,\, \dim L^{(1)}=1,\,\dim L^{(2)}=0.
$$
}
\end{note}

Combining the results of this section with Proposition \ref{gam}
and Theorem \ref{lami} we obtain the following theorem.

\begin{theorem}\label{41} Suppose that $\ell=2$, that $m\geq 2$, and that $f$ is
non-degenerate, symmetric and non-alternating. Then

(1) The $L(f)$-module $\gl(V)$ has $m+2$ composition factors. A
composition series can be obtained by inserting $m-1$ arbitrary
subspaces between $L(f)$ and $L(f)^{(1)}$ in the series
$$
0\subset L(f)^{(1)}\subset L(f)\subset \gl(V).
$$
Moreover, if $m=3$ or $m\geq 5$ then $L(f)^{(1)}$ is a simple Lie
algebra of dimension~${{m}\choose{2}}$.

(2) $L(f)$ is isomorphic to the symmetric square $S^2(V)$ as
$L(f)$-modules. Moreover, there is a basis of $V$ relative to
which $f$ has Gram matrix $D=\mathrm{diag}(d_1,\dots,d_m)$ and,
relative to this basis, $L(f)$ consists of all $A\in\gl(m)$ such
that $d_iA_{ij}=d_j A_{ji}$.

(3) $L(f)^{(1)}$ is isomorphic to the exterior square
$\Lambda^2(V)$ as $L(f)$-modules. Moreover, relative to the above
basis, $L(f)^{(1)}$ consists of all $A\in\gl(m)$ such that
$A_{ii}=0$ and $d_iA_{ij}=d_j A_{ji}$.

(4) $\gl(V)/L(f)\cong L(f)^{(1)}$ as $L(f)$-modules. In
particular, $\gl(V)$ has $m$ trivial composition factors, and 2
composition factors isomorphic to $L(f)^{(1)}\cong \Lambda^2(V)$,
which is itself the trivial module if and only if $m=2$.
\end{theorem}

\section{A composition series for the $\sy(2n)$-module $\gl(2n)$ when
$\ell=2$}\label{lastsec}

We assume throughout this section that $m=2n$ and that $f$ is
non-degenerate and alternating. We also assume that $\ell=2$,
except in Proposition \ref{modhn} and Note~\ref{ntabu}, where
$\ell$ is arbitrary.

\begin{theorem}\label{teosec12}
Suppose $m>4$. If $4|m$ then $L(f)^{(2)}/\s$ is a simple Lie
algebra. If $4\nmid m$ then $L(f)^{(2)}$ is a simple Lie algebra.
\end{theorem}

\begin{proof}
Let $J$ as defined in (\ref{defjo}) and let $I$ be a nonzero ideal of $L=L(J)^{(2)}$. Let $T$ and $C$ be the subspaces of $L$ defined by
$$
T = \left\lbrace \begin{pmatrix}
0 & t \\
0 & 0
\end{pmatrix} : t \text{ alternating } \right\rbrace,\quad C = \left\lbrace \begin{pmatrix}
0 & 0 \\
t & 0
\end{pmatrix} : t \text{ alternating } \right\rbrace.
$$

Suppose first that $I$ consists only of diagonal matrices. If $I$ contains a non-scalar matrix, the action of $\sl(n)$ on $I$ yields $u\oplus u'\in I$ for all $u\in \sl(n)$, against the assumption. So we have $I=\s$ if $4|m$, and a contradiction if $4\nmid m$.

Suppose next that $I$ contains a non-diagonal matrix $x$. Let $i,j,k$ be distinct indices between 1 and $n$. Since
$$
[e_{n+j,k}+e_{e_{n+k,j}},x]_{ik} = x_{i,n+j} \quad \text{and} \quad [e_{k,n+j}+e_{j,n+k},x]_{ki} = x_{n+j,i},
$$
we can assume that $x_{ij}\ne 0$. Since
$$
\mathrm{ad}(e_{jk}+e_{n+k,n+j})\circ \mathrm{ad}(e_{ji}+e_{n+i,n+j}) \circ \mathrm{ad}(e_{ki}+e_{n+i,n+k})(x) = x_{ij}(e_{ji}+e_{n+i,n+j}),
$$
it follows $e_{ji}+e_{n+i,n+j}\in I$. So the action of $\sl(n)$ on
$I$ shows that $u\oplus u'\in I$ for all $u\in \sl(n)$. Taking
$A=e_{31}$ and $B=e_{12}+e_{21}$ in (\ref{req3}) we obtain a
nonzero element of $T$, hence $T\subset I$ by Theorem \ref{ugas}.
Similarly we see that $C\subset I$. Therefore $I=L$.
\end{proof}

\begin{prop}\label{modhn} Let
$\mathfrak{h}(n)$ be the Heisenberg algebra of dimension $2n+1$,
whose underlying vector space is $V\oplus F$, with bracket
$$
[u+\beta,v+\gamma]=f(u,v).
$$
Let $u_1,\dots,u_n,v_1,\dots,v_n$ be a basis of $V$ relative to
which $f$ has Gram matrix $J$, as defined in (\ref{defjo}). Let
$z=1\in\mathfrak{h}(n)$. Let $F[X_1,\dots,X_n]$ be the polynomial
algebra over $F$ in $n$ commuting variables $X_1,\dots,X_n$. For
$q\in F[X_1,\dots,X_n]$ let $m_q$ be the linear endomorphism
``multiplication by $q$" of $F[X_1,\dots,X_n]$. Let
$0\neq\alpha\in F$. Then $F[X_1,\dots,X_n]$ becomes a faithful
$\mathfrak{h}(n)$-module via
$$
u_i\mapsto \partial/\partial X_i,\, v_i\mapsto \alpha\cdot
m_{X_i},z\mapsto \alpha\cdot I_{F[X_1,\dots,X_n]}.
$$
Moreover, $F[X_1,\dots,X_n]$ is irreducible if and only if
$\ell=0$. Furthermore, if $\ell$ is prime then
$(X_1^\ell,\dots,X_n^\ell)$ is an $\mathfrak{h}(n)$-invariant
subspace of $F[X_1,\dots,X_n]$ and
$F[X_1,\dots,X_n]/(X_1^\ell,\dots,X_n^\ell)$ is a faithful
irreducible $\mathfrak{h}(n)$-module of dimension $\ell^n$.
\end{prop}

\begin{proof} This is clear.
\end{proof}

\begin{lemma}\label{lmodl2} $L(f)/L(f)^{(2)}\cong \mathfrak{h}(n)$ as Lie
algebras.
\end{lemma}

\begin{proof} Let $J\in\gl(4)$ be defined as in (\ref{defjo}) and identify $L$
with $M(J)$. Consider the elements $a,b_1,\dots,b_n,c_1,\dots,c_n$
of $L(f)$ defined as follows in terms of $n\times n$ blocks:
$$
a=\left(%
\begin{array}{cc}
  e_{11} & 0 \\
  0 & e_{11} \\
\end{array}%
\right), b_i=\left(%
\begin{array}{cc}
  0 & e_{ii} \\
  0 & 0 \\
\end{array}%
\right),
c_i=\left(%
\begin{array}{cc}
  0 & 0 \\
  e_{ii} & 0 \\
\end{array}%
\right).
$$
The canonical projection of these elements produces a basis
$L(f)/L(f)^{(2)}$, with multiplication table:
\begin{equation}
\label{tablita} [b_i,c_i]=a.
\end{equation}
 Note the use of the matrix
description of $L(f)^{(2)}$, given in Theorem \ref{lami}, for the
computation of this table. It follows from (\ref{tablita}) that
$L(f)/L(f)^{(2)}\cong \mathfrak{h}(n)$.
\end{proof}

\begin{prop}\label{heiw2} Suppose $m=4$. Then the derived series
of $L=L(f)$ satisfies
$$
\dim L=10,\, \dim L^{(1)}=6,\,\dim L^{(2)}=5,\,\dim L^{(3)}=1,\,
\dim L^{(4)}=0.
$$
Here $L^{(3)}=\s$, $L^{(2)}\cong \mathfrak{h}(2)\cong L/L^{(2)}$,
and $U=L^{(2)}/L^{(3)}$ is a 4-dimensional abelian Lie algebra
that is irreducible as $L$-module. More precisely, the action of
$L$ on $U$ leaves invariant a non-degenerate alternating bilinear
form $g:U\times U\to F$, the kernel of the associated
representation $R:L\to L(g)$ is~$L^{(2)}$, the corresponding
4-dimensional faithful representation $S:L/L^{(2)}\to \gl(U)$ is
irreducible, and $S(L/L^{(2)})=R(L)$ is a 5-dimensional subalgebra
of the symplectic Lie algebra $L(g)$ for which the natural module
$U$ is irreducible. Moreover, if we identify $L/L^{(2)}$ with
$\mathfrak{h}(2)$ then $U$ is isomorphic to the module
$F[X_1,X_2]/(X_1^2,X_2^2)$ of Proposition \ref{modhn}, where
$\alpha=1$.
\end{prop}

\begin{proof} The dimensions of the terms of the derived series as
well as the fact that $L^{(3)}=\s$ follow from Theorem \ref{lami}.
Let $J\in\gl(4)$ be defined as in (\ref{defjo}) and identify $L$
with $M(J)$. Consider the following basis elements of $L^{(2)}$,
described in terms of $2\times 2$ blocks:
$$
x=\left(%
\begin{array}{cc}
  0 & e_{12}+e_{21} \\
  0 & 0 \\
\end{array}%
\right), y=\left(%
\begin{array}{cc}
  0 & 0 \\
  e_{12}+e_{21} & 0 \\
\end{array}%
\right),
$$
$$
e=\left(%
\begin{array}{cc}
  e_{12} & 0 \\
  0 & e_{21} \\
\end{array}%
\right),f=\left(%
\begin{array}{cc}
  e_{12} & 0 \\
  0 & e_{12} \\
\end{array}%
\right), z=\left(%
\begin{array}{cc}
  I_2 & 0 \\
  0 & I_2 \\
\end{array}%
\right).
$$
Their multiplication table is
$$
[x,y]=z=[e,f].
$$
Thus $L^{(2)}$ is a 5-dimensional Heisenberg algebra and
$L^{(2)}/L^{(3)}$ is a 4-dimensional abelian Lie algebra.

The bracket $[~,~]:L^{(2)}\times L^{(2)}\to L^{(3)}$ is
$L$-invariant by the Jacobi identity and the fact that
$L^{(4)}=0$. Since $L^{(2)}\cong \mathfrak{h}(5)$, the radical of
the alternating form $[~,~]$ is $L^{(3)}$. This induces a
non-degenerate $L$-invariant alternating form, say $g$, on
$U=L^{(2)}/L^{(3)}$. The matrix of $g$ with respect to the basis
${\mathcal B}=\{e+\s,x+\s,f+\s,y+\s\}$ of $U$ is simply $J$. By
definition of $L^{(3)}$, it follows that $L^{(2)}$ is in the
kernel of the representation $R:L\to L(g)\subset \gl(U)$, which
gives rise to a representation $S:L/L^{(2)}\to \gl(U)$. Let
$a,b_1,b_2,c_1,c_2$ be the elements of $L$ defined in
Lemma~\ref{lmodl2}. Then ${\mathcal
C}=\{a+L^{(2)},b_1+L^{(2)},b_2+L^{(2)},c_1+L^{(2)}, c_2+L^{(2)}\}$
is a basis of $L/L^{(2)}\cong \mathfrak{h}(2)$. The matrices
corresponding to these basis vectors via $S$ in $L(g)$ (identified
with $M(J)$) are given in terms of $2\times 2$ blocks as follows:
$$
R(a)=\left(%
\begin{array}{cc}
  I_2 & 0 \\
  0 & I_2 \\
\end{array}%
\right), R(b_1)=\left(%
\begin{array}{cc}
  0 & e_{12}+e_{21} \\
  0 & 0 \\
\end{array}%
\right),
R(b_2)=\left(%
\begin{array}{cc}
  e_{21} & 0 \\
  0 & e_{12}\\
\end{array}%
\right),
$$
$$
R(c_1)=\left(%
\begin{array}{cc}
  0 & 0 \\
  e_{12}+e_{21} & 0 \\
\end{array}%
\right), R(c_2)=\left(%
\begin{array}{cc}
  e_{12} & 0 \\
  0 & e_{21} \\
\end{array}%
\right).
$$
Since these matrices are linearly independent, it follows that the
kernel of $R$ is precisely $L^{(2)}$.

Let $W$ be a non-zero subspace of $U$ is invariant under these
matrices. We claim that $W=U$. Indeed, if $W$ contains any vector
from ${\mathcal B}$ then, acting through $b_1,b_2,c_1,c_2$, we see
that it contains them all. Suppose, if possible, that $W$ contains
no vectors from ${\mathcal B}$. Since $R(b_1)$ is nilpotent, $W$
must contain a non-zero vector from the nullspace of $R(b_1)$,
which is 2-dimensional and spanned by $e+\s,x+\s$. Thus $W$ must
have a vector of the form $(e+\s)+\alpha(x+\s)$ for some $0\neq
\alpha\in F$. The same reasoning applied to $R(b_2)$ shows that
$W$ must have a vector of the form $\alpha (x+\s)+\beta(f+\s)$ for
some $\beta\in F$, so $(e+\s)+\beta(f+\s)\in W$. Applying $b_2$,
it follows that $x+\s\in W$, a contradiction. This proves our
claim.

Consider the basis of $F[X_1,X_2]/(X_1^2,X_2^2)$ associated to
$X_2,1,X_1,X_1X_2$. Relative to this basis the matrices that the
endomorphisms $\partial/\partial X_1$, $\partial/\partial X_2$,
$m_{X_1}$ and $m_{X_2}$ induce on $F[X_1,X_2]/(X_1^2,X_2^2)$ are
exactly $R(b_1),R(b_2),R(c_1), R(c_2)$. This shows that $U$ arises
from the ${\mathfrak h(2)}$-module of Proposition \ref{modhn} by
taking $\alpha=1$.
\end{proof}

\begin{note}\label{ntabu}{\rm Let $d(n)$ be the smallest dimension
of a faithful ${\mathfrak h(n)}$-module. It is well-known that
$d(n)\leq n+2$. In fact, \cite{Bu} proves $d(n)=n+2$ if $\ell=0$.
Note that if $\ell=2$ then ${\mathfrak h(1)}\cong\sl(2)$, so
$d(1)=2$ in this case. In all other cases, i.e., if $(\ell, n)\neq
(2,1)$, then $d(n)=n+2$, as shown in \cite{S}.

We refer the reader to \cite{Bu} and \cite{CR} for the problem of
finding the smallest dimension of a faithful module for various
Lie algebras in characteristic 0.

Note, on the other hand, that ${\mathfrak h(n)}$ has no faithful
{\em irreducible} module if $\ell=0$.

More generally, let $L$ be a Lie algebra such that $[L,L]\cap
Z(L)\neq 0$ and suppose that $T:L\to\gl(U)$ is a faithful
irreducible representation. Then $\ell|\dim(U)$.

Indeed, let $z$ be a non-zero central element of $[L,L]\cap Z(L)$
and let $q\in F[X]$ be the minimal polynomial of $T(z)$. Since $U$
is irreducible and $z$ is central, we see that $q(X)$ is
irreducible. Let $W$ be a composition factor of the $L\otimes
K$-module $U\otimes K$, where $K$ is an algebraic closure of $F$.
Since $z$ is central, Schur's Lemma implies that $z$ acts through
a scalar operator on $W$. But $z\in [L,L]$, so the trace of this
operator is 0. Thus, the scalar is 0 or $\ell\mid \dim_K(W)$.  The
first case is impossible, for otherwise $0$ is a root of $q$,
whence $q=X$, i.e. $T(z)=0$, against the fact that $T$ is
faithful. It follows that $\ell$ divides the dimension over $K$ of
all composition factors of $U\otimes K$, whence $\ell|\dim(U)$.

When $\ell$ is prime then ${\mathfrak h(n)}$ certainly has
faithful irreducible modules. The smallest dimension for such
module is $p^n$. This is actually the only possible dimension when
$F$ is algebraically closed (but not otherwise). In fact, if $F$
is algebraically closed there is only one such module, up
isomorphism and an automorphism of ${\mathfrak h(n)}$, namely the
one described in Proposition \ref{modhn} with $\al=1$. See
\cite{S} for details.}
\end{note}

Combining Theorem \ref{teosec12}, Lemma \ref{lmodl2} and
Proposition \ref{heiw2} with Proposition~\ref{gam} and Theorem
\ref{lami} we obtain the following theorem.

\begin{theorem}\label{31} Suppose that $\ell=2$, that $m=2n$, and
that $f$ is non-degenerate and alternating. Then

(1) If $4\mid m$ then the $L(f)$-module $\gl(V)$ has $m+6$
composition factors. A composition series can be obtained by
inserting $m-1$ arbitrary subspaces between $L(f)$ and
$L(f)^{(1)}$ in the series
$$
0\subset \s\subset  L(f)^{(2)}\subset  L(f)^{(1)}\subset
L(f)\subset U \subset \sl(V)\subset \gl(V),
$$
where $U=L(f)\oplus\langle x\rangle$, $x\in\sl(V)$,
$[x,L(f)]\subseteq L(f)$, and
$$
x=\left(%
\begin{array}{cc}
  I_n & 0\\
  0 & 0 \\
\end{array}%
\right).
$$
All composition factors are trivial, except for
$L(f)^{(2)}/\s/\cong \sl(V)/U$, which has dimension
${{m}\choose{2}}-2$. Moreover, $L(f)^{(2)}/\s$ is a simple Lie
algebra if and only if $m>4$.

(2) If $m\neq 2$ and $4\nmid m$ then the $L(f)$-module $\gl(V)$
has $m+4$ composition factors.  A composition series can be
obtained by inserting $m-1$ arbitrary subspaces between $L(f)$ and
$L(f)^{(1)}$ in the series
$$
0\subset  L(f)^{(2)}\subset  L(f)^{(1)}\subset L(f)\subset
\sl(V)\subset \gl(V).
$$
All composition factors are trivial, except for $L(f)^{(2)}\cong
\sl(V)/L(f)$.

Moreover, $L(f)^{(2)}$ is a simple Lie algebra of dimension
${{m}\choose{2}}-1$.

(3) $L(f)$ is isomorphic to the symmetric square $S^2(V)$ as
$L(f)$-modules, and, relative to suitable basis of $V$, consists
of all matrices
\begin{equation}\label{maint44}
\left(\begin{array}{cc}
   A & B\\
   C & A'
  \end{array}
\right),\, A,B,C\in\gl(n),\text{ where }B,C\text{ are symmetric}.
\end{equation}

(4) $L(f)^{(1)}$ is isomorphic to the exterior square
$\Lambda^2(V)$ as $L(f)$-modules, and consists of all matrices
(\ref{maint44}) such that $B,C$ are alternating.

(5) $L(f)^{(2)}$ is isomorphic to the kernel of the contraction
$L(f)$-epimorphism $\Lambda^2(V)\to F$, given by $v\wedge w\mapsto
f(v,w)$, and consists of all matrices (\ref{maint44}) such that
$B,C$ are alternating and $\tr(A)=0$.

(6) $L(f)/L(f)^{(2)}$ is isomorphic, as Lie algebra, to
${\mathfrak h}(n)$, the Heisenberg algebra of dimension $2n+1$.
\end{theorem}

\begin{note}\label{uurr}{\rm
If $m=2$ then $L=L(f)$ is nilpotent of class 2, where
$$
\dim L=3,\, \dim L^{1}=1,\,\dim L^{2}=0.
$$
In particular, $\gl(V)$ has 4 trivial composition factors as
$L(f)$-module.

Note that $L\cong L/L^{(2)}\cong {\mathfrak h(1)}$ and that
through this identification the natural module for $L$ becomes the
module $F[X_1]/(X_1^2)$ of Proposition \ref{modhn} with
$\alpha=1$. }
\end{note}

\noindent{\bf Acknowledgments.} We thank J. Humphreys for useful
comments. This paper is dedicated to him, in appreciation for his
book \cite{H}, from which we learnt the subject.

\end{document}